\documentclass{amsart}
\usepackage{graphics}
\usepackage{graphicx}
\usepackage{amsfonts,amssymb,amsopn,amscd,amsmath}
 \numberwithin{equation}{section}
\usepackage{mathrsfs}
\usepackage[final]{changes}

\newtheorem{thm}{Theorem}[section]

\newtheorem{lem}[thm]{Lemma}{\rm}
\newtheorem{assumption}[thm]{Assumption}{\rm}
\newtheorem{prop}[thm]{Proposition}
\newtheorem{rem}[thm]{Remark}
\newtheorem{ex}[thm]{Example}
\newtheorem{defi}[thm]{Definition}

\def\x{\mathbf{x}}
\def\v{\mathbf{v}}
\def\u{\mathbf{u}}
\def\a{\mathbf{a}}
\def\p{\mathbf{p}}
\def\o{\mathbf{\omega}}
\def\X{\mathbf{X}}
\def\Z{\mathbf{Z}}
\def\W{\mathbf{W}}

\def\P{\mathbf{P}}
\def\M{\mathbf{M}}
\def\B{\mathbf{B}}
\def\T{\mathbf{T}}
\def\y{\mathbf{y}}
\def\z{\mathbf{z}}
\def\om{\mathbf{\Omega}}
\def\R{\mathbb{R}}
\def\A{\mathbf{A}}

\def\bS{\mathbf{S}}

\def\N{\mathbb{N}}
\def\K{\mathbf{K}}

\def\os{\underline{\sigma}}

\def\ss{\overline{\sigma}}
\def\bm{\boldsymbol{m}}
\def\bsigma{\boldsymbol{\sigma}}

\title[Representation of distributionally robust chance-constraints]{Distributionally robust polynomial chance-constraints under mixture ambiguity sets}

\author{Jean B. Lasserre}
\thanks{*The work of the two authors was supported by the European Research Council (ERC) via an ERC-Advanced Grant for the \# 666981 project TAMING}
\address{LAAS-CNRS and Institute of Mathematics\\
University of Toulouse\\
LAAS, 7 avenue du Colonel Roche\\
31077 Toulouse C\'edex 4, France\\
Tel: +33561336415}
\email{lasserre@laas.fr}
\author{Tillmann Weisser}
\address{LAAS-CNRS\\
University of Toulouse\\
LAAS, 7 avenue du Colonel Roche\\
31077 Toulouse C\'edex 4, France\\
Tel: +33561336441}
\email{tweisser@laas.fr}

\date{}
\begin{document}

\begin{abstract}
Given $\X\subset\R^n$, $\varepsilon \in (0,1)$, a parametrized family of probability distributions $(\mu_{\a})_{\a\in\A}$ on $\om\subset\R^p$, 
we consider the feasible set $\X^*_\varepsilon\subset\X$ associated with  the {\em distributionally robust} chance-constraint 
\[\X^*_\varepsilon\,=\,\{\x\in\X:\:{\rm Prob}_\mu[f(\x,\o)\,>\,0]> 1-\varepsilon,\,\forall\mu\in\mathscr{M}_\a\},\]
where $\mathscr{M}_\a$ is the set of all possibles mixtures of distributions $\mu_\a$, $\a\in \A$.
For instance and typically, the family
$\mathscr{M}_\a$ is the set of all mixtures of
Gaussian distributions on $\R$ with mean and standard deviation $\a=(a,\sigma)$ in some compact set $\A\subset \R^2$.
We provide a sequence of inner approximations $\X^d_\varepsilon=\{\x\in \X:w_d(\x) <\varepsilon\}$, $d\in\N$, 
where $w_d$ is a polynomial of degree $d$ whose
vector of coefficients is an optimal solution of a semidefinite program.
The size of the latter increases with the degree $d$. We also obtain the strong 
and highly desirable asymptotic guarantee that 
$\lambda(\X^*_\varepsilon\setminus \X^d_\varepsilon)\to0$
as $d$ increases, where $\lambda$ is the Lebesgue measure on $\X$. Same results
are also obtained for the more intricated case of distributionally robust ``joint" chance-constraints.
\end{abstract}
\maketitle

\section{Introduction}

\subsection*{Motivation} In many optimization 
and control problems\deleted{the} uncertainty is often modeled by a noise $\o\in\om\subset\R^p$ (following some probability distribution $\mu$), 
which interacts with the decision variable of interest\footnote{As quoted from R. Henrion, {\em the biggest challenge  from the algorithmic and theoretical points of view arise in chance constraints where the random and decision variables cannot be decoupled.
{\tt https://www.stoprog.org/what-stochastic-programming}
}}                                                                                                                                                                                                                                                                                                                     $\x\in\X\subset\R^n$ via some feasibility
constraint of the form $f(\x,\o)>0$ for some function $f:\X\to\R$. 
In the {\em robust} approach one imposes the constraint $\x\in\X^R:=\{\x\in\X: f(\x,\o)>0,\:\forall\o\in\om\}$ on the decision variable $\x$.
However, sometimes the resulting set $\X^R$ of robust decisions can be quite small or even empty.

On the other hand, 
if one knows the probability distribution $\mu$ of the noise $\o\in\om$, then a 
more appealing {\em probabilistic} approach is to tolerate a violation of the feasibility constraint 
$f(\x,\o)>0$, provided that
this violation occurs with small probability $\varepsilon>0$, fixed {\em \`a priori}. That is, one imposes the less conservative
chance-constraint ${\rm Prob}_\mu(f(\x,\o)\,>\,0)\,>\,1-\varepsilon$, which results
in the larger ``feasible set" of decisions
\begin{equation}
\label{X-mu}
\X^\mu_\varepsilon:=\{\x\in\X: {\rm Prob}_\mu(f(\x,\o)\,>\,0)\,>\,1-\varepsilon\,\}.\end{equation}
In both the robust and probabilistic cases, handling $\X^R$ or $\X^\mu_\varepsilon$ can be quite challenging and one is interested 
in respective approximations that are easier to handle.  
There is a rich literature on chance-constrained programming since Charnes and Cooper \cite{charnes}, Miller \cite{miller}, and
the interested reader is referred to Henrion \cite{henrion,henrion2}, Dabbene \cite{calafiore1}, Li et al. \cite{li}, 
Pr\'ekopa \cite{prekopa} and Shapiro \cite{shapiro} for a general overview
of chance constraints in optimization and control. Of particular interest are uncertainty models and methods that
allow to define tractable approximations of $\X^\mu_\varepsilon$. 
Therefore an important issue is to analyze under which conditions on $f$, $\mu$ and the threshold $\varepsilon$,
the resulting chance constraint (\ref{X-mu}) defines a convex set $\X^\mu_\varepsilon$; see e.g. 
Henrion and Strugarek \cite{henrion2},
van Ackooij \cite{convexity}, Nemirovski and Shapiro \cite{nemirovski}, 
Wang et al. \cite{wang}, and the recent work of  van Ackooij and Malick \cite{convexity-2}.
For instance, in \cite{convexity-2} the authors consider joint chance-constraint ($f(\x,\o)\in\R^k$) and 
show that $\X^\mu_\varepsilon$ is convex for sufficiently small $\varepsilon$
if $\o$ is an elliptical random vector and $f$ is convex in $\x$.
A different approach to model chance constraints was considered by the first author in \cite{lass-ieee}. It uses 
the general \replaced{Moment-Sums-of-Squares (Moment-SOS)}{Moment-SOS} methodology described in \cite{book1}. Related earlier work by
Jasour et al. \cite{lagoa1} have also used this Moment-SOS approach to solve
some control problems with probabilistic constraints.

However, one well-sounded critic to these probabilistic approaches is that it relies on the knowledge 
of the exact distribution $\mu$ of the noise $\o$, which in many cases may not be  a realistic assumption.
Therefore modeling the uncertainty via a single {\em known} probability distribution $\mu$
is questionable, and may render the chance-constraint 
${\rm Prob}_\mu(f(\x,\o)>0)\,>\,1-\varepsilon$ in (\ref{chance3}) counter-productive.
It would rather make sense to assume a {\em partial knowledge} 
on the {\em unknown} distribution $\mu$ of the noise $\o\in\om$. 

To overcome this drawback, {\em distributionally robust} chance-constrained problems consider probabilistic constraints that must hold for
a whole family of distributions and typically the family is characterized by the support and first and second-order moments;
see for instance Delage and Ye \cite{delage}, Edogan and Iyengar \cite{erdogan}, and Zymler et al. \cite{zymler}.
For instance Calafiore and El Ghaoui \cite{calafiore2} have shown
that when $f$ is bilinear then a tractable characterization via second-order cone constraints is possible.
Recently Yang and Xu \cite{yang} have considered non linear optimization problems 
where the constraint functions are concave in the decision variables
and quasi-convex in the uncertain parameters. They show that such problems are tractable if the 
uncertainty is characterized by its mean and variance only; in the same spirit see also 
Chao Duan et al. \cite{opf}, Chen et al. \cite{chen}, Hanasusanto et al. \cite{hanasusanto,hanasusanto2},
Tong et al. \cite{jogo}, Wang et al. \cite{wang}, Xie and Ahmed \cite{xie,xie2} and Zhang et al. \cite{zhang} for other tractable formulations of distributionally robust chance-constrained for optimal power flow problems.

\subsection*{The uncertainty framework} We also consider a framework where only partial knowledge of the uncertainty
is available. But instead of assuming knowledge of some moments like in e.g. \cite{yang},  
we assume that typically
the distribution $\mu$ can be {\em any mixture} of probability distributions $\mu_\a\in\mathscr{P}(\om)$
for some family $\{\mu_\a\}_{\a\in\A}\subset\mathscr{P}(\om)$ that depend on
a parameter vector $\a\in\A\subset\R^t$.
That is: 
\[\mu(B)\,=\,\int_\A \mu_\a(B)\,d\varphi(\a),\qquad \forall\,B\in\mathcal{B}(\om),\]
where $\varphi\in\mathscr{P}(\A)$ can be any probability distribution on $\A$. 
If $d\mu_\a(\o)=u(\o,\a)$ for some density $u$, then
by Fubini-Tonelli's Theorem \cite[Theorem p. 85]{doob}, the above measure $\mu$ is well-defined. For instance, for Value-at-Risk optimization (where $f$ is bilinear) El Ghaoui et al. \cite{el-ghaoui} 
suggest mixtures of Gaussian measures, see \cite[\S 1.3]{el-ghaoui}, and more generally,
families of probability measures with convex uncertainty sets on first and second-order moments\footnote{This more general uncertainty framework can also be analyzed 
with our approach (hence with arbitrary polynomial $f$).}.

Notice that in this framework {\em no} mean, variance or higher order moments have to be estimated.
Hence it can be viewed as an alternative and/or a complement to 
those considered in e.g. \cite{delage,calafiore2,erdogan,yang} when a good estimation
of such moments is not possible. Indeed in many cases,
providing a box (where the mean vector can lie)
and a possible range $\underline{\delta}\,\mathbf{I}\preceq \Sigma\preceq\overline{\delta}\,\mathbf{I}$ 
for the covariance matrix $\Sigma$, can be  more realistic than 
providing a single  mean vector and a single covariance matrix. Below are some examples of possible mixtures. In particular it should be noted that mixtures of Gaussian measures (Ex. \ref{ex-gauss}) are used in statistics precisely because of their high modeling power. Indeed they can approximate quite well a large family of distributions of interest in applications; see e.g. Dizioa et al. \cite{dizio},
Marron and Wand \cite{marron}, Wang et al. \cite{wang2}, and the recent survey of Xu \cite{survey}. Similarly, the family of measures on $\om$ 
with SOS-densities (Ex. \ref{mix-sos}) discussed in de Klerk et al. \cite{deklerk}  
also have a great modeling power with nice properties.

\begin{ex}[Mixtures of Gaussian's]
\label{ex-gauss}
$\om=\R$, $\a=(a,\sigma)\in \A:=[\underline{a},\overline{a}]\times [\os,\ss]\subset\R^2$, with 
$\underline{\sigma}>0$, and 
\[d\mu_\a(\o)\,=\,\frac{1}{\sqrt{2\pi}\sigma}\exp(-\frac{(\o-a)^2}{2\sigma^2})\,d\o,\]
that is $\mu$ is a mixture of Gaussian probability measures with mean-deviation couple 
$(a,\sigma)\in\A$.
\end{ex}
\begin{ex}[Mixtures of Exponential's]
\label{ex-expo}
$\om=\R_+$, $\a=a\in \A:=[\underline{a},\overline{a}]\subset\R$ with $\underline{a}>0$, and 
\[d\mu_\a(\o)\,=\,\frac{1}{a}\exp(-\o/a)\,d\o,\]
that is, $\mu$ is a mixture of exponential probability measures with parameter $1/a$, $a\in\A$.
\end{ex}
\begin{ex}[Mixtures of elliptical's]
\label{ex-elliptical}
In \cite{convexity-2} the authors have considered chance-constraints for the class of elliptical random vectors. In our framework and
in the univariate case,
$\om=\R$, $\a=(a,\sigma)\in \A:=[\underline{a},\overline{a}]\times [\os,\ss]\subset\R^2$, with 
$\underline{\sigma}>0$. Let $\theta:\R^+\to\R$ be such that $\int_0^\infty t^kd\theta(t)<\infty$ for all $k$,
and let $p=\int_\R\theta(t^2)dt$. Then:
\[d\mu_\a(\o)\,=\,\frac{1}{p\sigma}\theta\left(\frac{(\o-a)^2}{\sigma^2}\right)\,d\o,\quad\a\in\A.\]
\label{elliptical}
\end{ex}
\begin{ex}[Mixtures of Poisson's]
\label{ex-poisson}
$\om=\N$ and $\a=a\in\A:=[\underline{a},\overline{a}]\subset\R$ with $\underline{a}>0$, and 
\[\mu_\a(k)\,=\,\exp(-a)\,\frac{a^k}{k{\rm !}},\quad k=0,1,\ldots;\quad \a\in\A,\]
that is, $\mu$ is a mixture of Poisson probability measures with parameter $a\in\A$.
\end{ex}
\begin{ex}[Mixtures of Binomial's]
\label{ex-binomial}
Let $\om=\{0,1,\ldots,N\}$ and $\a=a\in\A:=[\underline{a},\overline{a}]\subset [0,1]$, and 
\[\mu_\a(k)\,=\,\binom{N}{k} a^k (1-a)^{N-k},\quad k=0,1,\ldots N.\]
\end{ex}
\begin{ex}
\label{ex-finite}
With $\om=\R$ one is given a finite family of probability measures $(\mu_i)_{i=1,\ldots,p}\subset\mathscr{P}(\om)$. Then $\a=a\in \A:=\{1,\ldots,p\}\subset\R$, and 
\[d\mu(\o)=\sum_{a=1}^p\lambda_a\,d\mu_a(\o);\quad\sum_{a\in \A} \lambda_a=1;\quad\lambda_a\geq0,\]
that is, $\mu$ is a finite convex combination of the probability measures $(\mu_\a)$.
\end{ex}
\begin{ex}[Mixtures of SOS densities]
\label{mix-sos}
Recently, de Klerk et al. \cite{deklerk} have introduced measures with SOS densities
to model some distributionally robust optimization problems because of their high modeling power.
This family also fits our framework. Let $\phi$ be a reference measure on $\om$ with known moments.
Then $\A=\Sigma[\x]_d$ and $d\mu_\a(\o)=p(\o)\,d\phi(\o)$ for some
$p\in\Sigma[\o]_d$ such that $\int_\om p\,d\phi=1$,
where $\Sigma[\o]_d$ is the space of SOS polynomials of degree at most $2d$. A measure 
$\mu_\a\in\mathscr{M}_\a$ is parametrized by the Gram matrix $\a=\X\succeq0$ of its density $p\in\Sigma[\o]_d$. For illustration purpose consider the univariate case $\om\subset\R$
and $p\in\Sigma[\o]_2$. Then 
\[p(\o)\,=\,\left[\begin{array}{c} 1 \\ \o \\\o^2\end{array}\right]^T\cdot
\left[\begin{array}{ccc} X_{11} & X_{12}& X_{13}\\
X_{12} & X_{22} & X_{23}\\
X_{13} & X_{23} & X_{33}\end{array}\right]\cdot
\left[\begin{array}{c} 1 \\ \o \\\o^2\end{array}\right]\cdot\]
and for every integer $k$, the moment $\int_\om \o^k\,d\mu_\a(\o)=
\int_\om \o^k\,p(\o)\,d\phi(\o)$
is a linear function of the 
coefficients of $p$ and hence of the parameter $\a=\X$. If $\M_\phi$ denotes the moment matrix of the reference measure $\phi$ on $\om$, then $\A=\{\X=\X^T\in\R^{3\times 3}: \X\succeq0;\:\langle \M_\phi,\X\rangle=1\}$ is a compact and convex basic semi-algebraic set.\footnote{Write the characteristic polynomial of $\X$ in the form $t\mapsto t^3-p_2(\X)\,t^2+p_1(\X)\,t-p_0(\X)$ with $p_2,p_1,p_0\in\R[\X]$. Then 
$\A=\{\X: p_i(\X)\geq0,\:i=0,1,2;\:\langle\M_\phi,\X\rangle=1\}$.} 
Additional bound constraints on moments are easily included as linear constraints on the coefficients of $p$.
\end{ex}
\deleted{In this uncertainty framework described by $\A$, one considers the ambiguity set:}

\deleted{where $\mathscr{P}(\A)$ is the set of probability measures on $\A$,
and in a {\em distributionally robust chance-constraint approach}, with $\varepsilon>0$ fixed,
one considers the set:}

\deleted{as new feasible set of decision variables. In general $\X^*_\varepsilon$ is non-convex
and can even be disconnected.  Therefore obtaining accurate approximations of $\X^*_\varepsilon$ is 
a difficult challenge. Our ultimate goal is to replace optimization problems in the general form:
}
All families $\mathscr{M}_\a$ described above share an important property
that is crucial for our purpose: Namely, {\em all} moments of $\mu_\a\in\mathscr{M}_\a$
are {\em polynomials} in the parameter $\a\in\A$. That is, for each $\beta$, $\int_{\om}\o^\beta d\mu_\a(\o)=p_\beta(\a)$
for some polynomial $p_\beta\in\R[\a]$.

\begin{rem}
\label{mu-sigma}
Another possible and related ambiguity set is to consider the family of measures $\mu$ on $\om$ whose
first and second-order moments $\a=(\bm,\bsigma)$ belong to some prescribed set $\A$.
The resulting ambiguity set has been already used in several contributions like e.g.
\cite{chen,el-ghaoui,hanasusanto,hanasusanto2,zymler}; however,
as discussed in \cite{deklerk}, the resulting ambiguity set might be too overly conservative 
(even in the case where $\A$ is the singleton $\{(\bm,\bsigma)\}$).
The methodology developed in this paper also applies with some ad-hoc modification; see \S 
\ref{newsection}.
\end{rem}

In this uncertainty framework one now has to consider the set:
\begin{equation}
\label{set-M}
\mathscr{M}_\a:=\{\int_\A d\mu_\a(\o)\varphi(d\a): \varphi\in\mathscr{P}(\A)\,\},
\end{equation}
where $\mathscr{P}(\A)$ is the set of probability measures on $\A$,
and in a {\em distributionally robust chance-constraint approach}, with $\varepsilon>0$ fixed,
one considers the set:
\begin{equation}
\label{new-set}
\X^*_\varepsilon\,:=\,\{\,\x\in \X: {\rm Prob}_\mu(f(\x,\o)\,>\,0)\,>\,1-\varepsilon,\quad  \forall \mu\in\mathscr{M}_\a\,\},
\end{equation}
as new feasible set of decision variables. In general $\X^*_\varepsilon$ is non-convex
and can even be disconnected.  Therefore obtaining accurate approximations of $\X^*_\varepsilon$ is 
a difficult challenge. Our ultimate goal is to replace optimization problems in the general form:
\begin{equation}
\label{chance3}
\displaystyle\min_\x\,\{\,u(\x): \x\in C\,\cap\,\X^*_\varepsilon\deleted{,}\},
\end{equation}
(where $u$ is a polynomial) with 
\begin{equation}
\label{chance4}
\displaystyle\min_\x\,\{\,u(\x): \x\in C\,\cap\,W\,\deleted{\geq\,0\,}\},
\end{equation}
where the uncertain parameter $\o$ has disappeared from the description
of a suitable inner approximation $W$ of $\X^*_\varepsilon$, with
$W:=\{\x\in\X: w(\x)\leq0\}$ 
for some polynomial $w$.
So if $C$ is a basic semi-algebraic set\footnote{\added{A basic semi-algebraic set of $\R^n$ is of the form $\{\x\in\R^n: g_j(\x)\geq0,\:j\in J\}$ for finitely many polynomials $(g_j)_{j\in J}$.}}
then (\ref{chance4}) is a standard polynomial optimization problem.
Of course the resulting optimization problem (\ref{chance4}) may still be hard to solve because 
the set $C\cap W$ is not convex in general. 
But this may be the price to pay for avoiding a too conservative formulation
of the problem.  However, since in \deleted{the} formulation (\ref{chance4}) one has got rid of the disturbance parameter $\o$, 
one may then apply the arsenal of Non Linear Programming algorithms to 
get a local minimizer of (\ref{chance4}). If $n$ is not too large
or if some sparsity is present in problem (\ref{chance4}) one may even run 
a hierarchy of semidefinite relaxations to approximate its global optimal value; for more details
on the latter, the interested reader is referred to \cite{book1}.

\subsection*{Contribution} 
We consider approximating $\X^*_\varepsilon$
as a challenging mathematical problem and explore whether it is possible to 
solve it under minimal assumptions on $f$ and/or $\mathscr{M}_\a$. Thus the focus is more on the existence
and definition of a prototype ``algorithm"  (or approximation scheme) rather than on its ``scalability". Of course the latter issue of scalability
is important for practical applications and we hope that the present contribution will provide insights on
how to develop  more ``tractable" (but so more conservative) versions. 

\replaced{O}{So o}ur contribution is {\em not} in the line of research concerned with ``tractable approximations" of (\ref{new-set}) (e.g. under some restrictions on $f$ and/or the family $\mathscr{M}_\a$) and should  be viewed
as complementary to previously cited contributions whose focus is on \deleted{on} scalability.

\replaced{G}{So, g}iven $\varepsilon>0$ fixed, a family $\mathscr{M}_\a$ (e.g. as in Examples \ref{ex-gauss}, \ref{ex-expo}, \ref{ex-elliptical}, \ref{ex-poisson}, \ref{ex-binomial}, \ref{ex-finite}) and an arbitrary polynomial $f$,
our main contribution is to provide rigorous and accurate inner approximations 
$(\X^d_\varepsilon)_{d\in\N}\subset\X^*_\varepsilon$ of
the set $\X^*_\varepsilon$, that converge 
to $\X^*_\varepsilon$ in a precise sense as $d$ increases. 
More precisely: \\

(i) We provide a nested sequence of {\it inner} approximations of the set $\X^*_\varepsilon$ in (\ref{new-set}),
in the form:
\begin{equation}
\label{chance2}
\X^d_\varepsilon\,:=\,\{\,\x\in \X:\: w_d(\x)\,<\,\varepsilon\,\},\qquad d\in\N,
\end{equation}
where $w_d$ is a polynomial of degree at most $d$, and $\X^{d}_\varepsilon\subset\X^{d+1}_\varepsilon\subset\X^*_\varepsilon$ for every $d$.\\

(ii) We obtain the strong and highly desirable asymptotic guarantee:
\begin{equation}
\label{volume}
\lim_{d\to\infty}\,\lambda(\X^*_\varepsilon\setminus\X^d_\varepsilon)\,=\,0,
\end{equation}
where $\lambda$ is the Lebesgue measure on $\X$. To the best of our knowledge it is the first result of this kind at this level of generality. 
Importantly, the ``volume" convergence (\ref{volume}) is obtained with
{\em no} assumption of convexity on the set $\X^*_\varepsilon$ (and indeed in general $\X^*_\varepsilon$ is not convex). \:

(iii) Last but not least, the same approach is valid with same conclusions
for the more intricate case of {\em joint} chance-constraints, that is, probabilistic constraints of the form
\[{\rm Prob}_\mu( f_j(\x,\o)>0,\quad j=1,\ldots,s_f)\,>\,1-\varepsilon,\quad\forall\mu\in\mathscr{M}_\a,\]
(for some given polynomials $(f_j)\subset\R[\x,\o]$). Remarkably, such constraints which are notoriously difficult to handle in general, are relatively easy to incorporate in our formulation.

We emphasize that our approach
is a non trivial extension 
of the numerical scheme proposed in
\cite{lass-ieee} for approximating $\X^*_\varepsilon$ when $\mathscr{M}_\a$ is a singleton (i.e.,
for standard chance-constraints).
\\

\subsection*{Methodology}
The approach that we propose for determining the set $\X^d_\varepsilon$
defined in (\ref{chance2}) is very similar in spirit to that in \cite{sirev} and \cite{lagoa1}, and a non trivial extension
of the more recent work \cite{lass-ieee} where only a single distribution $\mu$ is considered. It is an additional illustration of the versatility of the Generalized Moment Problem (GMP)
model and the moment-SOS approach outside the field of optimization.
Indeed we also define an 
infinite-dimensional LP problem $\P$ in an appropriate space of measures and a
sequence of semidefinite relaxations $(\P_d)_{d\in\N}$ of $\P$, whose associated monotone sequence of optimal values $(\rho_d)_{d\in\N}$ converges to the optimal value $\rho$ of $\P$. An optimal solution of this LP is a measure $\phi^*$ on $\X\times\om$. In
its disintegration $\hat{\phi^*}(d\o\vert\x)\,d\x$, the conditional probability $\hat{\phi^*}(d\o\vert\x)$ is a measure
$\mu_{\a(\x)}$, for some $\a(\x)\in\A$, which identifies the worst-case distribution at $\x\in\X$.

At an optimal solution of the dual 
of the semidefinite relaxation $(\P_d)$, we obtain a polynomial $w_d$ of degree $2d$ whose sub-level set
$\{\x\in\X:w_d(\x)<0\}$ is precisely the desired approximation $\X^d_\varepsilon$ of $\X^*_\varepsilon$ in (\ref{new-set}); in fact the sets
$(\X^d_\varepsilon)_{d\in\N}$ provide a sequence of {\it inner} approximations of $\X^*_\varepsilon$.

A\deleted{t last but not least and a}s in \cite{lass-ieee}, 
the support $\om$ of $\mu\in\mathscr{M}_\a$ and the set $\{(\x,\o): f(\x,\o)>0\}$ are {\em not} required to be compact, which includes the important case where $\mu$ can be a mixture of normal or exponential distributions.

As already mentioned, our methodology is {\em not} a straightforward extension of the work in \cite{lass-ieee} where
$\mathscr{M}_\a$ is a singleton and \replaced{the moments of this distribution are assumed to be known}{$\mu$ can be arbitrary provided that its moments are known}. Indeed in the present framework and in contrast to the singleton case treated in \cite{lass-ieee}, we do {\em not} know a sequence of moments because we do not know 
the exact distribution $\mu$ of the noise $\o$. 
For instance, some measurability issues (e.g. existence of a measurable selector) not present in \cite{lass-ieee}, arise. Also and in contrast to \cite{lass-ieee},
we cannot define outer approximations by passing to the complement of $\X^*_\varepsilon$.

Importantly, we also describe how to accelerate the convergence of our approximation scheme. It consists 
of adding additional constraints in our relaxation scheme, satisfied at every feasible solution
of the infinite-dimensional LP. These additional constraints come from a specific application of Stokes' theorem in the spirit of its earlier application in \cite{lass-ieee} but more intricate and not \added{as} a direct extension. Indeed, in the framework of our infinite-dimensional LP, it \added{is} required to define a measure with support on $\X\times\om\times\A$
(instead of $\X\times\om$ in \cite{lass-ieee}) which when passing to relaxations of the LP,
results in semidefinite programs of larger size (hence more difficult to solve). However,
this price to pay can be profitable because the resulting convergence is expected to be significantly faster (as experienced in other contexts).

\subsection*{Pros and cons}

On a positive side, our approach solves a difficult and challenging mathematical problem as
it provides 
a nested hierarchy of inner approximations $(\X^d_\varepsilon)_{d\in\N}$ of $\X^*_\varepsilon$ which converges
to $\X^*_\varepsilon$ as $d$ increases, a highly desirable feature. Also for every (and especially small) $d$, whenever not empty
the set $\X^d_\varepsilon$ is a valid (perhaps very conservative if $d$ is small) inner approximation of $\X^*_\varepsilon$ 
which can be exploited in applications if needed.

On a negative side, this methodology is computationally expensive, especially to obtain a very good inner approximation $\X^d_\varepsilon$
of $\X^*_\varepsilon$.  Therefore and so far, for accurate approximations
this approach is limited to relatively small size problems.
But again, recall that this approximation problem is a difficult challenge and at least, our approach
with strong asymptotic guarantees provides insights and indications on possible routes to follow
if one wishes to  scale the method to address larger size problems.
For instance, an interesting issue not discussed here is to investigate whether sparsity patterns already exploited in polynomial optimization (e.g. as in Waki et al. \cite{waki}) can be exploited in this context.

\section{Notation, definitions and preliminaries}
\label{notation}
\subsection{Notation and definitions}
Let $\R[\x]$ be the ring of polynomials in the variables $\x=(x_1,\ldots,x_n)$ and 
$\R[\x]_d$ be the vector space of polynomials of degree at most $d$
whose dimension is $s(d):=\binom{n+d}{n}$.
For every $d\in\N$, let  $\N^n_d:=\{\alpha\in\N^n:\vert\alpha\vert \,(=\sum_i\alpha_i)\leq d\}$, 
and
let $\v_d(\x)=(\x^\alpha)$, $\alpha\in\N^n_d$, be the vector of monomials, \replaced{i.e.,}{of} the canonical basis 
\replaced{$(\x^\alpha)_{|\alpha|\leq d}$}{$(\x^\alpha)$} of $\R[\x]_{d}$. 
A polynomial $p\in\R[\x]_d$ \replaced{can be}{is} written
\[\x\mapsto p(\x)\,=\,\sum_{\alpha\in\N^n}p_\alpha\,\x^\alpha\,=\,\langle \p,\v_d(\x)\rangle,\]
for some vector of coefficients $\p=(p_\alpha)\in\R^{s(d)}$. For a real symmetric matrix $\A$ the notation $\A\succeq0$ (resp. $\A\succ0$) stands for $\A$ is positive semidefinite (psd) (resp. positive definite (pd)). 
Denote by $\Sigma[\x]_d$
the convex cone of polynomials that are sums-of-squares (SOS) of degree at most $2d$, i.e.,
\[p\in\Sigma[\x]_d\quad\Longleftrightarrow\quad p(\x)\,=\,\sum_k p_k(\x)^2,\quad\forall \x\in\R^n,\]
for finitely many polynomials $p_k\in\R[\x]_d$. The convex cone 
$\Sigma[\x]_d$ is semidefinite representable. Indeed: $p\in\Sigma[\x]_d$ if and only if there exists a real symmetric metric $\X\succeq0$ of size $s(d)$ such that $p_\alpha=\langle\X,\B_\alpha\rangle$ for all $\alpha\in\N^n$, and where the symmetric matrices $(\B_\alpha)_{\alpha\in\N^n_{2d}}$ come from writing:
\[\v_d(\x)\,\v_d(\x)^T\,=\,\sum_{\alpha\in\N^n_{2d}}\x^\alpha\,\B_\alpha.\]
For more details the interested reader is referred to \cite{book1}.

Given a closed set $\mathcal{X}\subset\R^\ell$, denote by 
$\mathcal{B}(\mathcal{X})$ the Borel $\sigma$-field of $\mathcal{X}$,
$\mathscr{P}(\mathcal{X})$ the space of probability measures on $\mathcal{X}$
and by $\mathscr{B}(\mathcal{X})$ 
the space of bounded measurable 
functions on $\mathcal{X}$. We also denote by $\mathscr{M}(\mathcal{X})$ the space of finite 
signed Borel measures on $\mathcal{X}$ and by
$\mathscr{M}_+(\mathcal{X})$ its subset of finite (positive) measures 
on $\mathcal{X}$.\\

\noindent
{\bf Moment matrix.} Given a sequence $\y=(y_\alpha)_{\alpha\in\N^n}$, let $L_\y:\R[\x]\to\R$ be the linear (Riesz) functional
\[f\:(=\sum_\alpha f_\alpha\, \x^\alpha)\:\mapsto L_\y(f)\,:=\,\sum_\alpha f_\alpha\,y_\alpha.\]
Given $\y$ and $d\in\N$, the {\it moment} matrix associated with $\y$, is the 
real symmetric  $s(d)\times s(d)$ matrix $\M_d(\y)$ with rows and columns indexed in $\N^n_d$ and with entries 
\[\M_d(\y)(\alpha,\beta)\,:=\,L_\y(\x^{\alpha+\beta})\,=\,y_{\alpha+\beta},\quad \alpha,\beta\in\N^n_d.\]
Equivalently $\M_d(\y)=L_\y(\v_d(\x)\v_d(\x)^T)$ where $L_\y$ is applied entrywise.
\begin{ex} For illustration, consider the case $n=2$, $d=2$. Then:
	\[
	M_2(\y) = \left(\begin{array}{cccccc}
	y_{00} & y_{10} &y_{01} & y_{20} & y_{11} &y_{02}\\
	y_{10} & y_{20} &y_{11} & y_{30} & y_{21} &y_{12}\\
	y_{01} & y_{11} &y_{02} & y_{21} & y_{12} &y_{03}\\
	y_{20} & y_{30} &y_{21} & y_{40} & y_{31} &y_{22}\\
	y_{11} & y_{21} &y_{12} & y_{31} & y_{22} &y_{13}\\
	y_{02} & y_{12} &y_{03} & y_{22} & y_{13} &y_{04}\\
	\end{array}\right).
	\]
\end{ex}

A sequence $\y=(y_\alpha)_{\alpha\in\N^n}$ has a {\em representing measure} $\mu$ if $y_\alpha=\int \x^\alpha\,d\mu$ 
for all $\alpha\in\N^n$; if $\mu$ is unique then $\mu$ \added{is} said to be {\em moment determinate}.

A necessary condition for the existence of such a $\mu$ is that $\M_d(\y)\succeq0$ for all $d$.
\replaced{Except for in the univariate case $n=1$, this is only a necessary condition.}{But it is only a necessary condition (except in the univariate case $n=1$).} 
However the following sufficient condition in \cite[Theorem 3.13]{book1} is very useful:
\begin{lem}(\cite{book1})
\label{lem-carleman}
If $\y=(y_\alpha)_{\alpha\in\N^n}$ satisfies $\M_d(\y)\succeq0$ for all $d=0,1,\ldots$ and
\begin{equation}
\label{carleman}
\sum_{k=1}^\infty L_\y(x_i^{2k})^{-1/2k}\,=\,+\infty,\qquad i=1,\ldots,n,
\end{equation}
then $\y$ has a representing measure, and in addition $\mu$ is moment determinate.
\end{lem}
Condition (\ref{carleman}) due to Nussbaum is the multivariate generalization of
its earlier univariate version due to  Carleman; see e.g. \cite{book1}.\\

\noindent
{\bf Localizing matrix.} Given a sequence $\y=(y_\alpha)_{\alpha\in\N^n}$, and a polynomial $g\in\R[\x]$, the {\it localizing} moment matrix 
associated with $\y$ and $g$, is the 
real symmetric  $s(d)\times s(d)$ matrix $\M_d(g\,\y)$ with rows and columns indexed in $\N^n_d$ and with entries 
\begin{eqnarray*}
\M_d(g\,\y)(\alpha,\beta)&:=&L_\y(g(\x)\,\x^{\alpha+\beta})\\
&=&\sum_\gamma g_\gamma\,y_{\alpha+\beta+\gamma},\quad \alpha,\beta\in\N^n_d.
\end{eqnarray*}
Equivalently $\M_d(g\,\y)=L_\y(g(\x)\,\v_d(\x)\v_d(\x)^T)$ where $L_\y$ is applied entrywise.
\begin{ex} For illustration, consider the case $n=2$, $d=1$. Then the localization matrix associated 
with $\y$ and $g=x_1-x_2$, is:
	\[
	M_1(g\y) = \left(\begin{array}{ccc}
	y_{10}-y_{01} & y_{20}-y_{11}  & y_{11}-y_{02}\\
	y_{20}-y_{11} & y_{20}-y_{21}  & y_{21}-y_{12}\\
	y_{11}-y_{02} & y_{21}-y_{12}  & y_{12}-y_{02}\\
	\end{array}\right).
	\]
\end{ex}

\subsection*{Disintegration} Given a probability measure $\mu$ on a cartesian product $\mathcal{X}\times\mathcal{Y}$ of topological spaces, we may decompose $\mu$ into its marginal $\mu_\x$ on $\mathcal{X}$ and a stochastic kernel (or conditional probability measure) $\hat{\mu}(d\y\vert\cdot)$
on $\mathcal{Y}$ given $\mathcal{X}$, that is:
\begin{itemize}
	\item For every $\x\in\mathcal{X}$, $\hat{\mu}(d\y\vert\x)\in\mathscr{P}(\mathcal{Y})$, and
	\item For every $B\in\mathcal{B}(\mathcal{Y})$, the function $\x\mapsto 
	\hat{\mu}(B\vert\x)$ is measurable.
\end{itemize}

Then
\[\mu(A\times B)\,=\,\int_{A} \hat{\mu}(B\vert\x)\,\mu_\x(d\x),\quad\forall A\in\mathcal{B}(\mathcal{X}),
B\in\mathcal{B}(\mathcal{Y}).\]

\subsection{The family $\mathscr{M}_\a$}

Let $\om\subset\R^p$ be the ``noise" (or disturbance) space. Let $\A\subset\R^t$ be a compact set and 
for every $\a\in\A$, let $\mu_\a\in\mathscr{P}(\om)$.

\begin{assumption}
 \label{ass-on-Ma}
 The set $\{\mu_\a: \a\in\A\}\subset\mathscr{P}(\om)$ satisfies the 
 following:
 
\noindent
(i) For every $B\in\mathcal{B}(\om)$, the function $\a\mapsto \mu_\a(B)$
is measurable.

\noindent
(ii) For every $\beta\in\N$:
\begin{equation}
\label{def-pbeta}
\int_\om \o^\beta \,d\mu_\a(\o)\,=\,p_\beta(\a),\quad \forall\a\in\A,\end{equation}
for some polynomial $p_\beta\in\R[\a]$. 

\noindent
(iii) For every $\a\in\A$ and every polynomial $g\in\R[\o]$, $\mu_\a(\{\o: g(\o)=0\})=0$.

\noindent
(iv) For every bounded measurable (resp. bounded continuous) function $q$ on $\X\times\om$, the function
\[(\x,\a)\mapsto Q(\x,\a)\,:=\,\int_\om q(\x,\o)\,d\mu_\a(\o),\]
is bounded measurable (resp. bounded continuous) on $\X\times\A$.
 \end{assumption}
For instance if $d\mu_\a(\o)=\theta(\a,\o)\,d\o$ for some measurable density $\theta(\a,\cdot)$ on $\om$, then Assumption \ref{ass-on-Ma}(i) follows from Fubini-Tonelli's Theorem \cite{doob}, \added{moreover} Assumption \ref{ass-on-Ma}(iii) is also satisfied. Assumption \ref{ass-on-Ma}(ii) is \deleted{also} satisfied
in all Example\added{s} \ref{ex-gauss}-\ref{mix-sos}, as well as in
their multivariate extensions.
Assumption \ref{ass-on-Ma}(iv) is \deleted{also} satisfied in Example \ref{ex-gauss}, \ref{ex-expo}, 
\ref{elliptical}, \ref{ex-finite}, \ref{mix-sos}, and their natural multivariate extensions. For instance, in Example \ref{ex-gauss}
 where $\mathscr{M}_\a$ is the set of all possible mixtures of univariate Gaussian probability distributions  with mean and \replaced{standard}{standrad} deviation $(a,\sigma)\in\A$,
 the function
 \[Q(\x,\a):=\int_\om q(\x,\o)\,d\mu_\a\,=\,\frac{1}{\sqrt{2\pi}}\,\int_\om q(\x,(\sigma\,\o+a))\,\exp(-\o^2/2)\,d\o,\]
 is bounded measurable (resp. continuous) in $\a\in\A$ whenever $q$ is bounded measurable 
 (resp. continuous) on $\X\times\om$. 
 
 If the disturbance space $\om$ is non-compact, we need in addition:
 
 \begin{assumption}
 \label{ass-on-Ma-2}
 (If $\om\subset\R^p$ is unbounded):
 
\noindent
 There exists $c,\gamma>0$ such that for every $i=1,\ldots,p$:
 \begin{equation}
 \label{suff-carleman}
 \sup_{\a\in\A}\int_\om \exp(c\,\vert \o_i\vert)\,d\mu_\a(\o)\,<\,\gamma.\end{equation}
 \end{assumption}
  
\replaced{Note that this}{This} assumption is \deleted{also} satisfied in Example \ref{ex-gauss}, \ref{ex-expo}, \ref{elliptical}, \ref{ex-finite}, \ref{mix-sos}, and their natural multivariate extensions.

\begin{defi}\added{The set}
$\mathscr{M}_\a\subset\mathscr{P}(\om)$ is the space of all possible mixtures of probability measures $\mu_\a$, $\a\in\A$. That is, $\mu\in\mathscr{M}_\a$ if and only if there exists $\varphi\in\mathscr{P}(\A)$ such that
 \begin{equation}
\label{def-Ma}
\mu(B)\,=\,\int_\A \mu_\a(B)\,d\varphi(\a),\quad \forall B\in\mathcal{B}(\om).
\end{equation}
(By Assumption \ref{ass-on-Ma}(i), $\mu$ is well-defined.) In particular $\mu_\a\in\mathscr{M}_\a$ for all $\a\in\A$.
\end{defi}

 Let $\X\subset\R^n$ and $\A\subset\R^t$ be \deleted{a} compact basic semi-algebraic sets and let $\lambda$ be the Lebesgue measure on $\X$, scaled to be a probability measure on $\X$. We assume that $\X$ is simple enough so that all moments of $\lambda$ are easily calculated or available in closed-form. Typically $\X$ is a box, a simplex, an ellipsoid, etc. 
 The set $\om\subset\R^p$ is also a basic semi-algebraic set not necessarily compact (for instance it can
 be $\R^p$ or the positive orthant $\R^p_+$).

\begin{defi}
\label{def-extension}
Given a measure $\psi$ on $\X\times\A$ define
$\psi'$ on $\X\times\A\times\om$ by:
\[\psi'(B,C,D)\,=\,\int_{B\times C}\mu_\a(D)\,d\psi(\x,\a),\quad 
B\in\mathcal{B}(\X), \:C\in\mathcal{B}(\A),\:D\in\mathcal{B}(\om),\]
which is well defined by Assumption \ref{ass-on-Ma}(i). 
Its marginal $\psi'_{\x,\a}$ on $\X\times\A$ is $\psi$ and $\hat{\psi}'(\cdot\vert\x,\a)=\mu_\a$, for all $(\x,\a)\in\X\times\A$.
\end{defi}

Recall that $\mathscr{B}(\X\times \om)$ is the space of bounded 
measurable  functions on $\X\times\om$.
Define the linear mapping $\T:\mathscr{B}(\X\times \om)\to\mathscr{B}(\X\times \A)$ by:
\begin{equation}
\label{def-T}
g\mapsto \T g(\x,\a)\,:=\,\int_\om g(\x,\o)\,d\mu_\a(\o),\quad \a\in\A,\end{equation}
which is well-defined by Assumption \ref{ass-on-Ma}(iv). 
Therefore one may define  the {\em adjoint} linear mapping $\T^*:\mathscr{M}(\X\times\A)\to
 \mathscr{M}(\X\times\om)$ by:
 \begin{eqnarray*}
 \langle g,\T^*\psi\rangle &&\left(\,=\,\int_{\X\times\om}g(\x,\o)\,d\T^*\psi(\x,\o)\right)\,=\, \langle \T g,\psi\rangle\\
 &&\left(\,=\,\int_{\X\times\A}\left(\int_\om g(\x,\o)\,d\mu_\a(\o)\right)d\psi(\x,\a)\right),
 \end{eqnarray*}
 for all $g\in\mathscr{B}(\X\times \om)$ and all $\psi\in\mathscr{M}(\X\times \A)$.
 
\begin{lem}
\label{def-tstar}
Let $\T$ be as \added{in} (\ref{def-T}). Then for every $\psi\in\mathscr{M}_+(\X\times\A)$, $\T^*\psi=\psi'_{\x,\o}$ where $\psi'_{\x,\o}$ is the marginal on $\X\times\om$
of the measure $\psi'\in\mathscr{M}_+(\X\times\A\times\om)$ in Definition \ref{def-extension}. \end{lem}
\begin{proof}
Let $g\in\mathscr{B}(\X\times\om)$ \added{be} fixed\deleted{, arbitrary}. Then:
\begin{eqnarray*}
\langle g,\T^*\psi\rangle=\langle \T g,\psi\rangle&=&
\int_{\X\times\A}\left(\int_\om g(\x,\o)\,\mu_\a(d\o)\right)d\psi(\x,\a)\\
&=&\int_{\X\times\A}\left(\int_\om g(\x,\o)\,\hat{\psi'}(d\o\vert\x,\a)\right)d\psi(\x,\a)\\
&=&
\int_{\X\times\A\times\om}g(\x,\o)\,d\psi'(\x,\a,\o)\\
&=&\int_{\X\times\om}g(\x,\o)\,d\psi'_{\x,\o}(\x,\o)\,=\,
\langle g,\psi'_{\x,\o}\rangle.
\end{eqnarray*}
As this holds for every $g\in\mathscr{B}(\X\times\om)$, it follows
that $\T^*\psi=\psi'_{\x,\o}$.
\end{proof}

\begin{lem}
\label{T-poly}
Let Assumption \ref{ass-on-Ma} hold. Then the mapping $\T$ in (\ref{def-T}) extends to polynomials. Moreover, 
$\T( \R[\x,\o])\,\subset \R[\x,\a]$ and
\begin{equation}
\label{poly}
\langle h,\T^*\psi\rangle\,=\,\langle\T h,\psi\rangle,\quad\forall h\in\R[\x,\o],\:\forall\psi\in\mathscr{M}_+(\X\times\A).
\end{equation}
\end{lem}
\begin{proof}
Let $h\in\R[\x,\o]$ be fixed, arbitrary and write
$h(\x,\o)=\sum_{\alpha\beta}h_{\alpha\beta}\,\x^\alpha\,\o^\beta$. 
Then by Assumption \ref{ass-on-Ma}(ii):
\[\T h(\x,\a)\,:=\,
\sum_{\alpha,\beta} 
h_{\alpha\beta}\,\x^\alpha\,\int_\om\o^\beta\,d\mu_\a\,=\,
\sum_{\alpha,\beta} 
h_{\alpha\beta}\,\x^\alpha\,p_\beta(\a)\in\R[\x,\a].\]
Hence $\langle\T h,\psi\rangle=\int_{\x\times\A}\sum_{\alpha,\beta}
h_{\alpha\beta}\x^\alpha\,p_\beta(\a)\,d\psi(\x,\a)$.
Next, recalling the definition of $\psi'$ and $\psi'_{\x,\o}$:
\begin{eqnarray*}
\int_{\X\times\A}\sum_{\alpha,\beta}h_{\alpha\beta}\,\x^\alpha\,p_\beta(\a)\,d\psi(\x,\a)&=&\int_{\X\times\A}
 \left(\int_\om h(\x,\o)\mu_\a(d\o)\right)\,d\psi(\x,\a)\\
 &=&\int_{\X\times\om\times\A} h(\x,\o)\,d\psi'(\x,\o,\a)\\
 &=&
\int_{\X\times\om} h(\x,\o)\,d\psi'_{\x,\o}(\x,\o)\\
&=&\int_{\X\times \om} h(\x,\o) \,d\T^*\psi\,=\,\langle h,\T^*\psi\rangle,
\end{eqnarray*}
which yields (\ref{poly}).
\end{proof}

\section{An ideal infinite-dimensional LP problem}
\label{lp-formulation}

Let $\X\subset\R^n$, $\om\subset\R^p$ and $\A\subset\R^t$ be basic semi-algebraic sets. The sets $\X$ and $\A$ are assumed to be compact.
The ambiguity set $\mathscr{M}_\a$ is defined in \eqref{set-M}.
Let $f\in\R[\x,\o]$ be a given polynomial and with  $\varepsilon>0$ fixed, 
consider the set $\X^*_\varepsilon$ defined in (\ref{set-M}). Let
\begin{eqnarray}
\label{set-k}
\K&:=&\{(\x,\o)\in\X\times\om: f(\x,\o)\leq0\}\\
\label{set-kx}
\K_\x&:=&\{\o\in\om: (\x,\o)\in\K\,\},\quad \x\in\X.
\end{eqnarray}

\subsection{Basic idea and link with \cite{lass-ieee}}
Suppose for the moment that $\mathscr{M}_\a$ is the singleton $\{\mu\}$.
To approximate the set $\X^*_\varepsilon$ in \eqref{new-set} from inside,
the basic idea in \cite{lass-ieee} is to consider the infinite-dimensional LP:
\begin{equation}
\label{lp-ieee}
\rho\,=\,\displaystyle\sup_{\phi\in\mathscr{M}_+(\K)}\,\{\phi(\K):\:\phi\,\leq\,\lambda\otimes\mu\,\}.
\end{equation}
A dual of \eqref{lp-ieee} is the infinite-dimensional LP:
\begin{equation}
\label{lp-dual-ieee}
\rho^*\,=\,\displaystyle\inf_{w\in\R[\x,\o]}\,\{\,\int_{\X\times\Omega}w(\x,\o)\,d(\lambda\otimes\mu):\:
\mbox{$w\geq1$ on $\K$; $w\geq0$ on $\X\times\Omega$}\}.
\end{equation}
It is proved  in \cite{lass-ieee} that \eqref{lp-ieee} has a unique optimal solution
$d\phi^*=1_\K(\x,\o)d(\mu\otimes\lambda)$ and $\rho=\rho^*=\int_\X\mu(\K_\x)\lambda(d\x)$. 
In addition, for every feasible solution $w\in\R[\x,\o]$, let $h\in\R[\x]$ be the polynomial $\x\mapsto 
\int_\om w(\x,\o)\,d\mu(\o)$. Since $w\geq0$ on $\X\times\Omega$ and $w\geq1$ on $\K$ then $h(\x)\geq\mu(\K_\x)$ for all $\x\in\X$
and therefore $\{\x: h(\x)<\varepsilon\}\subset\{\x: \mu(\K_\x)<\varepsilon\}=\X^*_\epsilon$. 
Further, \cite{lass-ieee} defines a hierarchy of semidefinite relaxations of \eqref{lp-ieee} such that optimal solutions of their associated semidefinite duals provide polynomials $h_d\in\R[\x]_{2d}$ of increasing degree $d$ \replaced{with the property}{, such} that $\lambda(\X^*_\varepsilon\setminus\{\x: h_d(\x)<\varepsilon\})\to0$ a $d\to\infty$. This was possible because {\em one knows exactly and in closed form} all moments 
\deleted{$(\lambda_\alpha\cdot\mu_\beta)$} of the product measure $\lambda\otimes\mu$ on $\X\times\Omega$, a crucial ingredient of the semidefinite relaxations (14) in \cite{lass-ieee}.

When $\mathscr{M}_\a$ is not the singleton $\{\mu\}$, \replaced{a similar approach as the one above}{the above approach} is quite more involved because:
\begin{itemize}
\item One does {\em not} know the exact distribution of the noise $\o$; it can be any 
mixture of $(\mu_a)_{\a\in\A}$.
\item For each $\x\in\X$ such that $\K_\x\neq\emptyset$ one needs to identify
the worst-case distribution $\mu_\x\in\mathscr{M}_\a$, i.e., $\mu_{\a(\x)}$ where
$\a(\x)=\arg\max_{\a\in\A}\mu_\a(\K_\x)$.
\end{itemize}
\subsection*{Brief informal sketch of the \added{s}trategy}
To identify the worst-case distribution $\mu_\x$, we introduce an unknown distribution $\psi$ on $\X\times\A$ with marginal $\lambda$ on $\X$. It can be disintegrated into $\hat{\psi}(d\a\vert\x)\lambda(d\x)$ and
the goal is to compute $\psi^*=\hat{\psi}^*(d\a\vert\x)\lambda(d\x)$, with 
$\hat{\psi}^*(d\a\vert\x)=\delta_{\a(\x)}$. But recall that to provide numerical approximations, somehow
we need to access the moments of the involved measures, i.e., we need access to
$\gamma_\beta(\a)=\int_\om \o^\beta\,d\mu_\a(\o)$ for all $\beta\in\N^p$, and all $\a\in\A$. Fortunately, by Assumption
\ref{ass-on-Ma}(ii), $\gamma_\beta(\a)=p_\beta(\a)$ for some known polynomial $p_\beta\in\R[\a]$.
Then given $\psi$:
\begin{eqnarray*}
\underbrace{\int \x^\alpha\,p_\beta(\a)d\psi(\x,\a)}_{\mbox{moments of $\psi$}}&=&
\int_\X\x^\alpha\,\left(\int_\A p_\beta(\a)\,\hat{\psi}(d\a\vert\x)\right)\,\lambda(d\x)\\
&=&\int_\X\x^\alpha\,\left(\int_\A \left(\int_\Omega \o^\beta\,d\mu_\a(\o)\right)\,\hat{\psi}(d\a\vert\x)\right)\,\lambda(d\x)\\
&=&\int_\X\x^\alpha\,\left(\int_\Omega \o^\beta\mu(d\o) \right)\,\lambda(d\x)
\end{eqnarray*}
where $\mu(B)=\int_\A\mu_\a(B)\hat{\psi}(d\a\vert\x)$ for all $B\in\mathcal{B}(\Omega)$, i.e., $\mu\in\mathscr{M}_\a$. 
That is, by playing with $\psi$ (more precisely its conditional $\hat{\psi}(d\a\vert\x$)) one may explore for each $\x\in\X$, all possible mixtures of distributions
$\mu_\a$, $\a\in\A$. \replaced{I}{And i}mportantly, moments of such distributions are expressed in terms of moments of $\psi$. This is exactly what we need for computing approximations 
in the spirit of \cite{lass-ieee}. The linear mapping $\T$ in \eqref{def-T} 
is the tool to link measures $\phi$ on $\X\times\Omega$ with measures $\psi$ on $\X\times\A$.

As one may see, what precedes is a non trivial extension of the approach in \cite{lass-ieee}; in addition, the above informal derivation requires some measurability conditions, which by Lemma \ref{mu-star} below are guaranteed to hold.

\begin{lem}
\label{mu-star}
For each $\x\in\X$ there exist measurable 
mappings $\x\mapsto \a(\x)\in\A$ and $\x\mapsto \kappa(\x)\in\R$ such that:
\begin{equation}
\label{mu-star-1}
\kappa(\x)\,=\,\max\{\,\mu(\K_\x):\mu\in\mathscr{M}_\a\}\,=\,\max\{\,\mu_\a(\K_\x):\a\in\A\}\,=\,\mu_{\a(\x)}(\K_\x).
\end{equation}
\end{lem}
The proof is postponed to \S \ref{proof-mu-star}. Observe that for all $\x\in\X$, $\kappa(\x)=\mu_{\a(\x)}(\K_\x)=0$ whenever $\K_\x=\emptyset$.
Next, recall that for every $\x\in\X$:
\[\kappa(\x)\,=\,\max_\mu\,\{\mu(\K_\x):\mu\in\mathscr{M}_\a\}
\,=\,\max_\mu\,\{{\rm Prob}_{\mu}(f(\x,\o)\leq0):\mu\in\mathscr{M}_\a\}.\]
and so with $\K_\x^c:=\om\setminus\K_\x$,
\[\mu_{\a(\x)}(\K^c_\x)\,=\,\min_\mu\,\{\mu(\K^c_\x):\mu\in\mathscr{M}_\a\}
\,=\,\min_\mu\,\{{\rm Prob}_{\mu}(f(\x,\o)>0):\mu\in\mathscr{M}_\a\}.\]
\replaced{Consequently,}{So} the set $\X^*_\varepsilon$ in (\ref{new-set}) also reads:
\[\X^*_\varepsilon\,=\,\{\,\x\in\X: \mu_{\a(\x)}(\K^c_\x)\,>\,1-\varepsilon\,\}
\,=\,\{\,\x\in\X: \kappa(\x)\,<\,\varepsilon\,\}.\]

We next consider a certain infinite dimension LP with an important property:
Namely, any feasible solution of its dual (also an infinite dimensional LP)
provides the coefficients of some polynomial $w\in\R[\x]$ such that $\{\x: w(\x)<\varepsilon\}\subset \X^*_\varepsilon$; see Theorem \ref{th-abstract} below. It is the analogue in the 
present context of the LP (12) in \cite{lass-ieee}.

\subsection{An ideal infinite-dimensional LP}
\label{sec-infinite-lp}
Let $\lambda$ be the Lebesgue measure on $\X$, normalized to a probability measure, and 
consider the infinite-dimensional linear program (LP):
\begin{equation}
\label{new-primal-lp}
\begin{array}{rl}
\rho=\displaystyle\sup_{\phi,\psi}&
\{\,\phi(\K): \phi\leq\T^*\psi;\quad\psi_\x\,=\,\lambda,\\
&\phi\in \mathscr{M}_+(\K),\:\psi\in\mathscr{P}(\X\times\A)\,\},\end{array}
\end{equation}
where $\T^*$ is defined in Lemma \ref{def-tstar}.

\begin{thm}
\label{th1-lp}
The  infinite dimensional LP (\ref{new-primal-lp}) has optimal value $\rho=\int_\X \kappa(\x)\,d\x$. Moreover, the feasible pair $(\phi^*,\psi^*)$ with
\[d\phi^*(\x,\o):=1_\K(\x,\o)\,\mu_{\a(\x)}(d\o)\,d\lambda(\x),\quad d\psi^*(\x,\a)=\delta_{\a(\x)}(d\a)\,\lambda(d\x),\]
is an optimal solution of (\ref{new-primal-lp}), with $\x\mapsto(\a(\x),\kappa(\x))$ as in Lemma \ref{mu-star}.
\end{thm}
\begin{proof}
Let $(\phi,\psi)$ be an arbitrary feasible solution. 
Then
\begin{eqnarray*}
\phi(\K)=\langle 1,\phi\rangle &\leq&\langle 1_\K,\T^*\psi\rangle\,=\,\langle \T 1_\K,\psi\rangle\,=\,
\int_{\X\times\A}\mu_\a(\K_\x)\,d\psi(\x,\a)\\
&=&\int_{\X}\left(\int_\A\mu_{\a}(\K_\x)\,\hat{\psi}(d\a\vert\x)\right)\,\lambda(d\x)\\
&\leq&\int_{\X}\mu_{\a(\x)}(\K_\x)\,\lambda(d\x)\quad\mbox{[by Lemma \ref{mu-star}]}\\
&=&\int_{\X}\kappa(\x)\,\lambda(d\x).
\end{eqnarray*}
Next let $d\phi^*(\x,\o)=1_\K(\x,\o)\,\mu_{\a(\x)}(d\o)\,\lambda(d\x)$ and
$d\psi^*(\x,\a)=\delta_{\a(\x)}(d\a)\,\lambda(d\x)$ with $\x\mapsto\a(\x)$
as in Lemma \ref{mu-star}.  Then $\phi^*\in\mathscr{M}_+(\K)$ 
and $\phi^*\leq\T^*\psi^*$. Moreover:
\[\phi^*(\K)\,=\,\int_\X\left(\int_{\K_\x}d\mu_{\a(\x)}(\o)\right)\,\lambda(d\x)\,=\,
\int_{\X}\kappa(\x)\,\lambda(d\x)\]
as $\mu_{\a(\x)}(\K_\x)=0$ whenever $\K_\x=\emptyset$.
\end{proof}
Hence, Theorem \ref{th1-lp} states that  in an optimal solution $(\phi^*,\psi^*)$ of the LP 
(\ref{new-primal-lp}), at every $\x\in \X$, the conditional probability $\hat{\phi}^*(\cdot\vert\x):=\mu_{\a(\x)}\in\mathscr{M}_\a$
identifies the {\em worst case} noise distribution $\mu_{\a(\x)}$ in $\mathscr{M}_\a$, that is,
the one which maximizes ${\rm Prob}_{\mu}(f(\x,\o)\leq0)$ over $\mathscr{M}_\a$,
hence which minimizes ${\rm Prob}_{\mu}(f(\x,\o)>0)$ over $\mathscr{M}_\a$.

\subsection*{A dual of (\ref{new-primal-lp})}

Recall that by Lemma \ref{T-poly},
the mapping $\T$ extends to polynomials, and so
consider the infinite dimensional LP:
\begin{equation}
\label{chance-lp-dual}
\begin{array}{rl}
\rho^*\,=\,\displaystyle\inf_{h,w}&\{\displaystyle\int_\X w\,d\lambda:\quad h(\x,\o)\,\geq\,1\quad\mbox{on $\K$},\\
&w(\x)-\T h(\x,\a)\geq0\quad \mbox{on $\X\times\A$},\\
&h\geq0\mbox{ on $\X\times\om$},\\
&w\in\R[\x];\:h\in\R[\x,\o]\}.
\end{array}
\end{equation}

\begin{thm}
\label{th-abstract}
The infinite dimensional LP (\ref{chance-lp-dual}) is a dual of (\ref{new-primal-lp})
that is, $\rho^*\geq\rho$. In addition, for every feasible solution $(w,h)$ of \eqref{chance-lp-dual}:
\begin{equation}
\label{aux1}
w(\x)\,\geq\,\kappa(\x),\qquad\forall \x\in\X,
\end{equation}
with $\x\mapsto \kappa(\x)$ as in Lemma \ref{mu-star}, and so for every $\varepsilon>0$:
\begin{equation}
\label{subset-0}
\X_w\,:=\,\{\x: w(\x)\,<\,\varepsilon\}\subset \{\x: \mu_{\a(\x)}(\K^c_\x)\,>\,1-\varepsilon\}\,=\,\X^*_\varepsilon.
\end{equation}
Moreover, suppose that 
there is no duality gap, i.e., $\rho^*=\rho$, and let
$(w_n,h_n)$ be a minimizing sequence of (\ref{chance-lp-dual}). Then
with $\Vert\cdot\Vert_1$ the norm of $L_1(\X,\lambda)$:
\begin{equation}
\label{subset}
\lim_{n\to\infty}\,\Vert w_n-\kappa\Vert_1\,=\,0,\quad
\mbox{and}\quad\displaystyle\lim_{n\to\infty}\lambda(\X^*_\varepsilon\setminus \X_{w_n})=0.
\end{equation}
\end{thm}
The proof is postponed to \S \ref{proof-th-abstract}. 
Theorem \ref{th-abstract} is still abstract and the challenge is to define a practical method to compute $\rho$ and to {\em construct} effectively, inner approximations of $\X^*_\varepsilon$ which converge to $\X^*_\varepsilon$, as the approximations $\X_{w_n}$ in (\ref{subset}).

\section{A hierarchy of semidefinite relaxations}

In this section we provide a numerical scheme to
approximate from above the optimal value $\rho$ of the infinite-dimensional LP (\ref{new-primal-lp}) and its dual 
(\ref{chance-lp-dual}). In addition, from an optimal solution of
the approximation of the dual (\ref{chance-lp-dual}),
one is able to construct {\em effectively} inner approximations of $\X^*_\varepsilon$
which converge to $\X^*_\varepsilon$, as the approximations \replaced{$\X_{w_n}$}{$H_n$} in (\ref{subset}).

As a preliminary, we first show that the measure
$\T\psi^*$ in the constraint $\phi\leq\T^*\psi$, \added{and} identified as
$\psi'_{\x,\o}$ in Lemma \ref{def-tstar}, can be ``handled" through its moments.
 
\begin{lem}
\label{lem-handle}
Let $\nu\in\mathcal{P}(\X\times\om)$ and
$\psi\in\mathcal{P}(\X\times\A)$ with $\psi_\x=\lambda$, and let
\begin{equation}
\label{moments}
\int_{\X\times\om}\x^\alpha\o^\beta\,d\nu(\x,\o)\,=\,\int_{\X\times\A}\x^\alpha\,p_\beta(\a)\,d\psi(\x,\a),\quad\forall\alpha\in\N^n,\beta\in\N^p.\end{equation}
Then $\nu=\T^*\psi$.
\end{lem}
The proof is postponed to \S \ref{proof-lem-handle}.

\subsection{A hierarchy of semidefinite relaxations of (\ref{new-primal-lp})}

The compact basic semi-algebraic sets $\X$ and $\A$ \added{and the basic semi-algebraic set $\Omega$} are defined by:
\begin{eqnarray}
\label{set-X}
\X&:=&\{\x:g_j(\x)\geq0,\:j=1,\ldots,m\}\\
\label{set-A}
\A &:=&\{\a\in\R^t: q_\ell(\a)\,\geq\,0,\quad \ell=1,\ldots,L\}\\
\label{set-omega}
\om&:=&\{\o\in\R^p: s_\ell(\o)\geq\,0,\quad \ell=1,\ldots,\bar{s}\},
\end{eqnarray}
for some polynomials $(g_j)\subset\R[\x]$, $(q_\ell)\subset\R[\a]$ and $(s_\ell)\subset\R[\o]$. In particular if \replaced{$\bar{s}=0$}{$s=0$} then $\om=\R^p$. 

Let $d_j:=\lceil {\rm deg}(g_j)/2\rceil$, $d'_\ell:=\lceil {\rm deg}(q_\ell)/2\rceil$, $j=1,\ldots,m$, 
$\ell=1,\ldots,L$. Also let $d^1_\ell=\lceil{\rm deg}(s_\ell)/2\rceil$, $\ell=1,\ldots,\bar{s}$.
For notational convenience we also define 
$g_{m+1}(\x,\o)=-f(\x,\o)$ with $d_{m+1}:=\lceil {\rm deg}(f)/2\rceil$.
As $\X$ and $\A$ are compact, \deleted{then} $\X\subset\{\x:\Vert\x\Vert^2\leq M\}$ and 
$\A\subset\{\a:\Vert\a\Vert^2\leq M\}$ for some $M$ sufficiently large.
Therefore with no loss of generality we may and will assume that
\begin{equation}
\label{archimedian}
g_1(\x)=M-\Vert\x\Vert^2,\quad q_1(\a)=M-\Vert\a\Vert^2,\end{equation}
for some $M$ sufficiently large. Similarly
if $\om$ in (\ref{set-omega}) is compact then we may and will also assume  
that $s_1(\o)=M-\Vert\o\Vert^2$.
This will be very useful as it ensures compactness of the feasible sets of 
the semidefinite relaxations defined below.

Next, recall that by Assumption \ref{ass-on-Ma}(ii), for every $\beta\in\N^p$,
$\displaystyle\int_\om \o^\beta d\mu_\a(\o)=p_\beta(\a)$, for all $\a\in\A$,
for some polynomial $p_\beta\in\R[\a]$.\\

Consider the following hierarchy of semidefinite programs indexed by $d\geq \replaced{d_{\mathrm{min}}}{d_0}\in\N$, \added{where $2d_{\mathrm{min}}$ is the largest degree appearing in the polynomials that describe $\K,\om,$ and $\A$:}
\begin{equation}
\label{chance-sdp}
\begin{array}{rl}
\rho_d=&\displaystyle\sup_{\y,\u,\v}y_{00}\\
\mbox{s.t.}& L_{\y+\u}(\x^\alpha\o^\beta)-L_\v(\x^\alpha p_\beta(\a))\,=0,\:\vert\alpha+{\rm deg}(p_\beta)\vert\leq 2d,\\
&L_\v(\x^\alpha)=\lambda_\alpha,\quad\alpha\in\N^n_{2d},\\
&\M_d(\y),\M_d(\u),\M_d(\v)\,\succeq\,0,\\
&\M_{d-d_{m+1}}(g_{m+1}\,\y)\succeq0,\\
&\M_{d-d_j}(g_j\,\y),\: \M_{d-d_j}(g_j\,\u),\,\M_d(g_j\,\v)\succeq0,\quad j=1,\ldots,m,\\
&\M_{d-d^1_\ell}(s_\ell\,\y),\:\M_{d-d^1_\ell}(s_\ell\,\u)\,\succeq 0,\quad \ell=1,\ldots\bar{s},\\
&\M_{d-d'_\ell}(q_\ell\,\v)\,\succeq0,\quad \ell=1,\ldots,L,
\end{array}
\end{equation}
where $\y=(y_{\alpha\beta})$, \replaced{$\u=(u_{\alpha\beta})$}{$\u=(v_{\alpha\beta})$}, $(\alpha,\beta)\in\N^n\times\N^p$ and
$\v=(v_{\alpha\eta})$, $(\alpha,\eta)\in\N^n\times\N^t$.

\begin{prop}\label{prop: relaxation}
The semidefinite program \eqref{chance-sdp} is a relaxation of \eqref{new-primal-lp}, i.e., $\rho_d\geq\rho$ for all $d\geq \replaced{d_{\mathrm{min}}}{d_0}$. In addition $\rho_{d+1}\leq\rho_d$ for all $d\geq \replaced{d_{\mathrm{min}}}{d_0}$.
\end{prop}
\begin{proof}
That $\rho_{d+1}\leq\rho_d$ is straightforward as more constraints are added as $d$ increases.
Next, let $\phi,\psi$ be an arbitrary feasible solution of \eqref{new-primal-lp} and let
$\y=(y_{\alpha\beta})$, \replaced{$\u=(u_{\alpha\beta})$}{$\u=(v_{\alpha\beta})$}, $(\alpha,\beta)\in\N^n\times\N^p$ and
$\v=(v_{\alpha\eta})$, $(\alpha,\eta)\in\N^n\times\N^t$ be the moments sequences of the measure
$\phi,\T^*\psi-\phi$ and $\psi$, respectively. 
In the following we first show that $\y, \u$, and $\v$ are feasible for \eqref{chance-sdp}.
Necessarily $\M_d(\y)\succeq0$ because for every vector $\p\in\R^{s(d)}$ (with $s(d):=\binom{n+d}{n}$),
\[\langle \p,\M_d(\y)\,\p\rangle\,=\,\int_\K p(\x,\o)^2d\phi(\x,\o)\,\geq\,0,\]
where $p\in\R[\x]_d$ has $\p$ as vector of coefficients. With similar arguments, \added{one can}
show that $\M_d(\u)\succeq0$ and $\M_d(\v)\succeq0$. Next, as $\phi$ is supported on $\K$,
\deleted{then} for every $\p\in\R^{s(d-d_j)}$:
\[\langle \p,\M_{d-d_j}(g_j\,\y)\,\p\rangle\,=\,\int_\K p(\x,\o)^2\,g_j(\x)\,d\phi(\x,\o)\,\geq\,0,\]
because $g_j(\x)\geq0$ whenever $(\x,\o)\in\K$. Hence $\M_d(g_j\,\y)\succeq0$ (and similarly $\M_{d-d_j}(g_j\,\u),\M_{d-d_j}(g_j\,\v)\succeq0$). Similarly for every $\p\in\R^{s(d-d^1_\ell)}$:
\[\langle \p,\M_{d-d^1_\ell}(s_\ell\,\y)\,\p\rangle\,=\,\int_\K p(\x,\o)^2\,s_\ell(\o)(\x)\,d\phi(\x,\o)\,\geq\,0,\]
because $s_\ell(\o)\geq0$ whenever $(\x,\o)\in\K$. Hence $\M_{d-d^1_\ell}(s_\ell\,\y)\succeq0$ (and similarly $\M_{d-d^1_\ell}(s_\ell\,\u)\succeq0$). Finally, as $\psi$ is supported on
$\X\times\A$ then for every $\p\in\R^{s(d-d'_\ell)}$:
\[\langle \p,\M_{d-d'_\ell}(q_\ell\,\v)\,\p\rangle\,=\,\int_{\X\times\A} p(\x,\o)^2\,q_\ell(\a)\,d\psi(\x,\a)\,\geq\,0,\]
because $q_\ell(\a)\geq0$ whenever $(\x,\a)\in\X\times\A$. Hence $\M_{d-d'_\ell}(q_\ell\,\v)\succeq0$.

Next, as $\psi_\x=\lambda$, then for every $\alpha\in\N^n_{2d}$:
\[L_\v(\x^\alpha\o^0)\,=\,v_{\alpha0}\,=\,\int_{\X\times\A} \x^\alpha\,d\psi(\x,\a)\,=\,\int_\X\x^\alpha \,d\lambda(\x)\,=\,\lambda_\alpha.\]
Finally, as $\T^*\psi\geq\phi$ (equivalently $\T^*\psi=\phi+\nu$ with $\nu$ being the slack measure $\T^*\psi-\phi$)), then for every $(\alpha,\beta)\in\N^n\times\N^p$ with
$\vert\alpha\vert+{\rm deg}(p_\beta)\leq2d$:
\begin{eqnarray*}
L_\v(\x^\alpha\,p_\beta(\a))&=&\int_{\X\times\A}\x^\alpha p_\beta(\a)\,d\psi(\x,a)
=\int_{\X\times\A}\T (\x^\alpha\o^\beta)\,d\psi(\x,\a)\\
&=&\int_{\X\times\Omega}\x^\alpha\o^\beta\,d(\T^*\psi)(\x,\o)\,=\,
\int_{\X\times\Omega}\x^\alpha\o^\beta\,d(\phi+\nu)\\
&=&L_{\y+\u}(\x^\alpha\o^\beta),
\end{eqnarray*}
and therefore $(\y,\u,\v)$ is a feasible solution of \eqref{chance-sdp}. Finally, we show that $\rho_d\geq \rho$. To that end, note that $\y_{00}=\int_\K d\phi=\phi(\K)$, i.e.,  $\rho_d\geq\phi(\K)$, and in particular 
$\rho\added{_d}\geq\phi^*(\K)\added{=\rho}$ for the optimal solution 
$\phi^*$ of \eqref{new-primal-lp}.
\end{proof}\replaced{By Proposition \ref{prop: relaxation}}{Therefore} (\ref{chance-sdp}) defines a {\em hierarchy} of semidefinite programs whose associated sequence
of optimal values $(\rho_d)_{d\geq \replaced{d_{\mathrm{min}}}{d_0}}$ is nonnegative and monotone non increasing. 

\subsection*{Size of \eqref{chance-sdp}} The semidefinite program involves 
moment matrices of size 
$\tau_1:=\binom{n+p\added{+d}}{d}$  and $\tau_2=\binom{n+t\added{+d}}{d}$. 
So even though \eqref{chance-sdp} can be solved in time polynomial in its input size,
the computational burden is rapidly prohibitive, and in view of the present status of semidefinite solvers,
this approach is limited to relaxations for\deleted{t} which $\max[\tau_1,\tau_2]\leq 10^3$, i.e., for modest 
size problems. However, recall that even the first relaxations provide inner approximations of $\X^*_\varepsilon$.

\subsection{The dual of the semidefinite relaxations (\ref{chance-sdp})}

The dual of (\ref{chance-sdp}) is an SDP which has the following high-level
interpretation in terms of SOS positivity certificates of size parametrized by $d$:

\begin{equation}
\label{chance-sdp-dual}
\begin{array}{rl}
\rho^*_d=\displaystyle\inf_{h,w,\sigma^i_j} &\displaystyle\int_\X w(\x)\,d\lambda(\x):\\
\mbox{s.t}&h(\x,\o)-1=\displaystyle\sum_{j=0}^{m+1}\sigma^1_j\,g_j
+\displaystyle\sum_{\ell=1}^{\bar{s}}\sigma^1_\ell\,s_\ell,\quad\forall (\x,\o);\\
&h(\x,\o)=\displaystyle\sum_{j=0}^m\sigma^2_j\,g_j
+\displaystyle\sum_{\ell=1}^{\bar{s}}\sigma^2_\ell\,s_\ell,\quad\forall (\x,\o);\\
&w(\x)-\displaystyle\sum_{\alpha,\beta}h_{\alpha\beta}\,\x^\alpha\,p_\beta(\a)=
\displaystyle\sum_{j=0}^m\sigma^3_j\,g_j+\displaystyle\sum_{\ell=1}^L\sigma^3_\ell\,q_\ell,\quad\forall (\x,\a);\\
&{\rm deg}(h),\:{\rm deg}(w)\leq 2d;\:\sigma_j^1\in\Sigma[\x,\o]_{d-d_j},j=0,\ldots,m+1,\\
&\sigma_j^2\in\Sigma[\x,\o]_{d-d_j},\:\sigma_j^3\in\Sigma[\x,\a]_{d-d_j};j=0,\ldots,m,\\
&\sigma^1_\ell,\sigma^2_\ell\in\Sigma[\o]_{d-d^1_\ell};\sigma_\ell^3\in\Sigma[\x,\a]_{d-d'_\ell},\:\ell=1,\ldots,L,
\end{array}
\end{equation}
where $h(\x,\o)=\sum_{\vert\alpha+\beta\vert\leq 2d}h_{\alpha\beta}\,\x^\alpha\,\o^\beta$, and $w(\x)=\sum_{\vert\alpha\vert\leq 2d}w_{\alpha}\,\x^\alpha$.

In compact form, (\ref{chance-sdp-dual}) is 
the high level interpretation of the dual SDP of \eqref{chance-sdp} in terms of SOS positivity certificates
of size parametrized by $d$. Indeed:

$\bullet$ The dual variable $h_{\alpha\beta}$ associated with the equality constraint
$L_{\y+\u}(\x^\alpha\o^\beta)=L_\v(\x^\alpha\,p_\beta(\a))$ is the coefficient of $\x^\alpha\o^\beta$ 
for the polynomial $h\in\R[\x,\o]_{2d}$ in \eqref{chance-sdp-dual}. 

$\bullet$ Similarly,  the dual variable $w_\alpha$
associated with the equality constraint  $L_\v(\x^\alpha)=\lambda_\alpha$  is the coefficient
of $\x^\alpha$ for the polynomial $w\in\R[\x]_{2d}$ in \eqref{chance-sdp-dual}.

$\bullet$ $\sigma_j^1\in\Sigma[\x,\o]_{d-d_j}$  (resp. $\sigma_j^2\in\Sigma[\x,\o]_{d-d_j}$) 
is the SOS polynomial associated with the 
matrix dual variable $\X^1_j\succeq0$ (resp. $\X^2_j\succeq0$)
associated with the semidefinite constraint $\M_{d-d_j}(g_j\,\y)\succeq0$ (resp. $\M_{d-d_j}(g_j\,\u)\succeq0$) of \eqref{chance-sdp}.

$\bullet$ $\sigma_\ell^1\in\Sigma[\x,\o]_{d-d^1_\ell}$  (resp. $\sigma_\ell^2\in\Sigma[\x,\o]_{d-d^1_\ell}$) 
is the SOS polynomial associated with the 
matrix dual variable $\Z^1_j\succeq0$ (resp. $\Z^2_j\succeq0$)
associated with the semidefinite constraint $\M_{d-d^1_\ell}(s_\ell\,\y)\succeq0$ (resp. $\M_{d-d^1_\ell}(s_\ell\,\u)\succeq0$) of \eqref{chance-sdp}.

$\bullet$ $\sigma_j^3\in\Sigma[\x,\a]_{d-d_j\ell}$  (resp. $\sigma_\ell^3\in\Sigma[\x,\a]_{d-d'_\ell}$) 
is the SOS polynomial associated with the 
matrix dual variable $\W^1_j\succeq0$ (resp. $\W^2_j\succeq0$)
associated with the semidefinite constraint $\M_{d-d_j}(g_j\,\v)\succeq0$ (resp. $\M_{d-d'_\ell}(s_\ell\,\v)\succeq0$) of \eqref{chance-sdp}.

The SDP \eqref{chance-sdp-dual} is a {\em reinforcement} of the infinite-dimensional dual (\ref{chance-lp-dual}) in which the positivity constraints 
have been replaced with SOS {\em positivity certificates} \`a la Putinar \added{(see \cite{book1})}. For instance,
the positivity constraint ``$h\geq1$ on $\K$" in (\ref{chance-lp-dual}) becomes in (\ref{chance-sdp-dual}) the stronger:
\[h(\x,\o)-1=\displaystyle\sum_{j=0}^{m+1}\sigma^1_j\,g_j
+\displaystyle\sum_{\ell=1}^{\bar{s}}\sigma^1_\ell\,s_\ell,\quad\forall (\x,\o),\]
for some SOS polynomials $(\sigma^1_j)$ \added{and $(\sigma^1_\ell)$}.
\begin{thm}
\label{th4}
Let Assumption \ref{ass-on-Ma} (and Assumption \ref{ass-on-Ma-2} as well if
$\om$ is unbounded) \added{hold} and assume that $\K,\A,\X\times\om$ and $\X\times\om\setminus\K$ all have nonempty interior. Then:

(i) Slater's condition holds for (\ref{chance-sdp}) and so strong duality holds. That is, for every $d\geq \replaced{d_{\mathrm{min}}}{d_0}$, $\rho^*_d=\rho_d$ and (\ref{chance-sdp-dual}) has an optimal solution
$(h_d,w_d)$. 

(ii) Next, define $\X^d_\varepsilon:=\{\x\in\X: w_d(\x)\,<\varepsilon\,\}$.
Then $\X^d_\varepsilon\subset\X^*_\varepsilon$ for every $d\geq \replaced{d_{\mathrm{min}}}{d_0}$. In addition, if $\lim_{d\to\infty}\rho_d
=\rho$ then with $\x\mapsto\kappa(\x)$ as in Lemma \ref{mu-star}:
\begin{equation}
\label{property-1}
\lim_{d\to\infty}
\Vert w_d(\x)-\kappa(\x)\Vert_{L_1(\X,\lambda)}\,=\,0\quad\mbox{and}\quad
\lim_{d\to\infty}
\lambda(\X^*_\varepsilon\setminus\X^d_\varepsilon)\,=\,0. 
\end{equation}
\end{thm}
A proof is postponed to \S \ref{proof-th4}.
Note that so far, optimal solutions of (\ref{chance-sdp-dual}) provide us
with a hierarchy of inner approximations $\X^d_\varepsilon\subset\X^*_\varepsilon$, $d\geq \replaced{d_{\mathrm{min}}}{d_0}$.
In addition, if the approximation scheme (\ref{chance-sdp}) \deleted{(or (\ref{chance-sdp-dual}))} is such that
$\lim_{d\to\infty}\rho_d=\rho$, then \replaced{Theorem}{Lemma} \ref{th4}
states that the inner approximations $(\X^d_\varepsilon)$ have the additional 
strong asymptotic property (\ref{property-1}) which in turn implies the highly desirable
convergence result (\ref{property-1}). \\

So to obtain (\ref{property-1}) we need to ensure that $\lim_{d\to\infty}\rho_d=\rho$ as $d\to\infty$.

\begin{thm}
\label{th3}
Let Assumption \ref{ass-on-Ma} hold, and if $\om$ is unbounded let
Assumption \ref{ass-on-Ma-2} \replaced{hold as well}{also hold}.
Consider the hierarchy of semidefinite programs (\ref{chance-sdp}) 
with associated monotone sequence of optimal values $(\rho_d)_{d\geq \replaced{d_{\mathrm{min}}}{d_0}}$. Then for each $d\geq \replaced{d_{\mathrm{min}}}{d_0}$ there is an optimal solution
$(\y^d,\u^d,\v^d)$ and 
\[\rho_d\,=\,y^d_{00}\downarrow \phi^*(\K)\,=\,\rho,\quad\mbox{as $d\to\infty$},\]
 where $\phi^*$ is part of an optimal solution $(\phi^*,\psi^*)$ of (\ref{new-primal-lp}).
\end{thm}
A proof is postponed to \S \ref{proof-th3}.

\subsection{Accelerating convergence via Stokes}

In previous works of a similar flavor but for volume computation in \cite{sirev} and \cite{lass-ieee,gaussian}, it was observed that the convergence 
$\rho_d\to\rho$ was rather slow. In our framework, by inspection of the dual (\ref{chance-lp-dual}), a
potential slow convergence may arise as one tries to approximate from above a discontinuous function
(the indicator function $1_\K$ of a compact set $\K$) by polynomials, and therefore one is faced with an annoying Gibb's phenomenon. 
The trick proposed in \cite{sirev} resulted in a significant acceleration of the convergence
but at the price of \replaced{losing}{loosing} its monotonicity (a highly desirable feature).  This motivated the other strategy proposed in \cite{lass-ieee,gaussian}, based on Stokes' theorem, which also resulted in a significantly faster convergence, but this time without  \replaced{losing}{loosing} its monotonicity.

In this section we provide a means to accelerate the convergence 
$\rho_d\to\rho$ in Theorem \ref{th3}, also based on Stokes' theorem applied to
the optimal solution $\phi^*$ of (\ref{new-primal-lp}). However, its implementation is 
much more complicated than in \cite{lass-ieee,gaussian} as it requires introducing
an additional measure in the LP (\ref{new-primal-lp}).

It works when the probability measures $(\mu_\a)_{\a\in\A}$, satisfy some additional property: For all $\a\in\A$, either, 

$\bullet$ $d\mu_\a(\o)=s(\o,\a)$ for some polynomial $s\in\R[\o,\a]$, or 

$\bullet$ $d\mu_\a(\o)=q_0(\a)\,s_1(\o,\a)\,\exp(s_2(\o,\a))d\o$, where 
$q_0$ is a rational function and
$\o\mapsto s_i(\o,\a)$ belongs to $\R[\o]$, $i=1,2$. In addition, for every fixed $\a$,
and $i=1,\ldots,p$,
$\partial s_2(\o,\a)/\partial \o_i$ is a rational function of $\o$.

We detail the procedure for the case where $\om=\R$,
$\a=(a,\sigma)\in \A:=[\underline{a},\overline{a}]\times [\os,\ss]\subset\R^2$,
and
\[d\mu_\a(\o)\,=\,\frac{1}{\sqrt{2\pi}\sigma}\,\exp(-\frac{(\o-a)^2}{2\sigma^2})\,d\o,\quad \a=(a,\sigma)\in\A,\]
i.e., $\mathscr{M}_\a$ is the family of all possible mixtures of Gaussian probability measures on $\R$ with mean $a\in [\underline{a},\overline{a}]$ and standard deviation $\sigma\in[\os,\ss]$. Recall that 
$\K_\x=\{\o\in\R: f(\x,\o)\leq0\,\}$.
For every $\a\in\A$, an\deleted{d} extended version of 
Stokes' theorem in \cite[Lemma 3.1, p. 151]{gaussian}, yields:

\[\int_{\K_\x}\frac{\partial\,(\o^\beta\,f(\x,\o)\,\exp(\frac{-(\o-a)^2}{2\sigma^2}))}{\partial \o}\,d\o
=0,\quad \beta=0,1,\ldots\]
for every $\x\in\X$. That is:
\begin{equation}
\label{aux-stokes}
\int_{\K_\x}q_\beta(\x,\o,\a)\,d\mu_{\a}(\o)\,=\,0,\quad \forall\x\in\X,\:\forall \beta=0,1,\ldots,
\end{equation}
where $q_\beta\in\R[\x,\o,\a]$ reads:
\[q_\beta(\x,\o,\a):=
\frac{\sigma^2\partial\,(\o^\beta\,f(\x,\o))}{\partial \o}-\o^\beta f(\x,\o)(\o-a).\]
This  in turn implies that for every $\a\in\A$,
\begin{equation}
\label{aux11}
\int_\X\x^\alpha\left(\a^\gamma\int_{\K_\x}q_\beta(\x,\o,\a)\,d\mu_{\a}(\o)\right)\,d\lambda(\x)\,=0,\end{equation}
for all $\alpha\in\N^n$, $\gamma\in\N^2$,  and $\beta=0,1,\ldots$. In particular,
with $\a(\x)$ as in Lemma \ref{mu-star}:
\[\int_\X\x^\alpha\,\a(\x)^\gamma\left(\int_{\K_\x}q_\beta(\x,\o,\a(\x))\,d\mu_{\a(\x)}(\o)\right)\,d\lambda(\x)\,=0,\]
or equivalently:
\begin{equation}
\label{stokes-phi-star}
\int_\K \x^\alpha\,\a(\x)^\gamma\,q_\beta(\x,\o,\a(\x))\,d\phi^*(\x,\o)\,=\,0,
\end{equation}
for all $\alpha\in\N^n$, $\gamma\in\N^2$, $\beta=0,1,\ldots$.

Therefore in the infinite-dimensional LP (\ref{new-primal-lp}) we may add the linear ``generalized" moment constraints
(\ref{stokes-phi-star}) because they are satisfied at an optimal solution
$\phi^*$. However, the function $(\x,\o)\mapsto \added{\a(\x)^\gamma}q_\beta(\x,\o,\a(\x))$ is {\em not} a 
polynomial and these constraints cannot be implemented directly. 

To overcome this problem, \added{we} introduce 
an additional  measure $\nu$ on $\K\times\A$ and impose that  its marginal $\nu_{\x,\o}$ on $\K$ 
is $\phi$ and its marginal $\nu_{\x,\a}$ on $\X\times\A$ is dominated by $\psi$. That is
$\nu_{\x,\o}=\phi$ and $\nu_{\x,\a}\leq\psi$. \replaced{On this newly introduced measure we now can impose}{We also need introduce} the (Stokes) moment constraints:
\begin{equation}
\label{add-nu}
\int_{\X\times\A\times\om}\x^\alpha\,\a^\gamma\,q_\beta(\x,\o,\a)\,d\nu(\x,\o,\a)\,=\,0,\quad
\forall (\alpha,\beta,\gamma)\in\N^n\times\N^p\times\N^2.
\end{equation} 
Recall that at an optimal solution 
$(\phi^*,\psi^*)$ of (\ref{new-primal-lp}),
$\hat{\psi}^*(d\a\vert\x)=\delta_{\a(\x)}$ for all $\x\in\X$ and so every feasible solution of the form
$(\phi^*,\psi^*,\nu)$ (hence with $\nu_{\x,\o}=\phi^*$ and $\nu_{\x,\a}\leq\psi^*$)
satisfies $\hat{\nu}_{\x,\a}(d\a\vert\x)=\delta_{\a(\x)}$ for all $\x\in{\rm supp}(\nu_\x)$, 
i.e., for all $\x\in{\rm supp}(\phi^*_\x)$.\footnote{This is because by compactness of $\X$ and $\A$, for all $\alpha,\gamma$,
$\int\x^\alpha\a^\gamma\,d\nu_{\x,\a}(\x,\a)=\int \x^\alpha\,\a(\x)^\gamma \eta(\x,\a(\x))\lambda(d\x)$ for some nonnegative measurable function $\eta\leq1$, and so $\nu_\x(d\x)=\eta(\x,\a(\x))\lambda(d\x)$. This in turn implies that for every $\gamma$ and almost all $\x\in\X$, $\int \a^\gamma\hat{\nu}_{\x,\a}(d\a\vert\x)=\a(\x)^\gamma=\int \a^\gamma \delta_{\a(\x)}(d\a)$. Hence $\hat{\nu}_{\x,\a}(d\a\vert\x)=\delta_{\a(\x)}(d\a)$ for all $\x\in{\rm supp}(\nu_\x)$.}
Disintegrating $\nu$ as $\hat{\nu}(d\a\vert \x,\o)\,d\nu_{\x,\o}(\x,\o)$ 
yields
\begin{eqnarray*}
0&=&\int_{\K}\x^\alpha\left( \int_\A \,\a^\gamma\,q_\beta(\x,\o,\a)\,\hat{\nu}(d\a\vert \x,\o)\right)\,d\nu_{\x,\o}(\x,\o),\\
&=&\int_{\K}\x^\alpha \,\a(\x)^\gamma\,q_\beta(\x,\o,\a(\x))\,d\nu_{\x,\o}(\x,\o),\\
&=&\int_{\K}\x^\alpha \,\a(\x)^\gamma\,q_\beta(\x,\o,\a(\x))\,d\phi^*(\x,\o),\quad
\forall \alpha\in\N^n,\,\gamma\in\N^2,\:\beta\in\N,
\end{eqnarray*}
which is (\ref{stokes-phi-star}). 
So in the semidefinite relaxation (\ref{chance-sdp}), we introduce the additional vector 
$\z^1=(z^1_{\alpha,\beta,\gamma})$ and $\z^2=(z^2_{\alpha,\gamma})$, 
$\alpha\in\N^n$, $\beta\in\N^p$, $\gamma\in\N^t$,
(ideally the respective moments of the measure $\nu$  and $\psi^*-\nu_{\x,\a}$ where
$\nu$ is described above), and the constraints
\begin{eqnarray*}
L_{\z^1}(\x^\alpha\,\o^\beta)&=&L_\y(\x^\alpha\,\o^\beta),\quad \forall\: \vert\alpha\vert+\vert\beta\vert\leq 2d,\\
L_{\z^1}(\x^\alpha\,\a^\gamma)+L_{\z^2}(\x^\alpha\,\a^\gamma)&=&L_\v(\x^\alpha\,\a^\gamma),
\quad\forall \:\vert\alpha\vert+\vert\gamma\vert\leq 2d,\\
L_{\z^1}(\x^\alpha\,\a^\gamma\,q_\beta(\x,\o\,\a))&=&0
\quad\forall \:\vert\alpha\vert+\vert\gamma\vert+{\rm deg}(q_\beta)\leq 2d.
\end{eqnarray*}
Of course, for the same index $d\in\N$, the resulting semidefinite relaxation is more computationally demanding as it now includes moments of the measure $\nu$
on $\X\times\A\times\om$ (whereas in (\ref{chance-sdp}) we only have moments of measures 
on $\X\times\om$ and $\X\times\A$). However, being more constrained its optimal value can be significantly smaller and the resulting convergence
$\rho_d\to\rho$ as $d$ increases, can be expected to be much faster.

Actually, in this framework of mixtures of Gaussian measures, one may also replace
(\ref{new-primal-lp}) with:
\begin{equation}
\label{new-primal-lp-stokes}
\begin{array}{rl}
\rho=\displaystyle\sup_{\phi,\psi\geq0}&
\{\,\phi(\K\times\A): \phi_{\x,\o}\leq\T^*\psi;\quad\phi_{\x,\a}\leq\psi;\quad\psi_\x\,=\,\lambda,\\
&\displaystyle\int_{\K\times\A} \x^\alpha\,\a^\gamma\,q_\beta(\x,\o,\a)\,d\phi(\x,\a,\o)=0,\\
&\\&\forall\alpha\in\N^n,\,\gamma\in \N^2,\,\beta\in\N^p;\\
&\phi\in \mathscr{M}_+(\K\times\A),\:\psi\in\mathscr{P}(\X\times\A)\,\},\end{array}
\end{equation}
where now $\phi$ is a measure on $\K\times\A$ (instead of $\K$ before). The dual of (\ref{new-primal-lp-stokes}) reads:
\begin{equation}
\label{dual-lp-stokes}
\begin{array}{rl}
\rho^*=\displaystyle\inf_{w,h,\theta,s}&\{\,\displaystyle\int_\X w(\x)\,d\lambda(\x):\\
&h(\x,\o)+\theta(\x,\a)\,\\
&+\displaystyle\sum_{\alpha,\beta,\gamma}s_{\alpha\gamma\beta}\,\x^\alpha\a^\gamma\,q_\beta(\x,\a,\o)\,\geq\,1,\quad\mbox{on $\K\times\A$}\\
&w(\x)-Th(\x,\a)-\theta(\x,\a)\,\geq\,0,\quad\mbox{on $\X\times\A$}\\
&\mbox{$h\geq0$ on $\X\times\om$};\quad \mbox{$\theta\geq0$ on $\X\times\A$},\\
&w\in\R[\x],\:h\in\R[\x,\o],\:\theta\in\R[\x,\a],\:s\in\R[\x,\a,\o]\:\}.
\end{array}
\end{equation}
In particular, let $(w,h,\theta,s)$ be an arbitrary  feasible solution of (\ref{dual-lp-stokes}).
Then for  every $\x\in\X$,
and with $\a(\x)$ as in Lemma \ref{mu-star}, 
\begin{eqnarray*}
w(\x)&\geq&Th(\x,\a(\x))+\theta(\x,\a(\x))=\int_\om h(\x,\o)d\mu_{\a(\x)}(d\o)+\theta(\x,\a(\x))\\
&=&\int_\om (h(\x,\o)+\theta(\x,\a(\x)))\,d\mu_{\a(\x)}(d\o)\\
&\geq&\int_{\K_\x}(h(\x,\o)+\theta(\x,\a(\x)))\,d\mu_{\a(\x)}(d\o)\\
&=&\displaystyle\int_{\K_\x} [\,\underbrace{h(\x,\o)+\theta(\x,\a(\x))+\displaystyle\sum_{\alpha,\beta,\gamma}s_{\alpha\beta\gamma}\,
\x^\alpha\a(\x)^\gamma\,q_\beta(\x,\a(\x),\o)}_{\mbox{$\geq1$ on $\K_\x$}}\\
&& -\displaystyle\sum_{\alpha,\beta,\gamma}s_{\alpha\beta\gamma}\,
\x^\alpha\a(\x)^\gamma\,q_\beta(\x,\a(\x),\o)\,]\,d\mu_{\a(\x)}(d\o)\\
&\geq &\int_{\K_\x}d\mu_{\a(\x)}-\displaystyle\sum_{\alpha,\beta,\gamma}s_{\alpha\beta\gamma}\,
\x^\alpha\a(\x)^\gamma\,\underbrace{\int_{\K_\x}q_\beta(\x,\a(\x),\o)\,d\mu_{\a(\x)}(d\o)}_{\mbox{$=0$ by (\ref{aux-stokes})}}\\
&\geq&\kappa(\x).
\end{eqnarray*}

One may show that in (\ref{new-primal-lp-stokes}) and (\ref{dual-lp-stokes}), 
\[\rho\,=\,\rho^*\,=\,\int_\X \kappa(\x)\,d\lambda(\x),\]
and (\ref{property-1}) in  Theorem \ref{th4} also holds for an arbitrary minimizing sequence
of (\ref{dual-lp-stokes}).

The hierarchy of associated SDP-relaxations 
(i.e. the analogues for \eqref{new-primal-lp-stokes} of 
\eqref{chance-sdp} for \eqref{new-primal-lp}) and
its dual hierarchy are obtained in an obvious manner by truncation of the infinite sequences. 
At step $d$ of the latter hierarchy we also obtain a polynomial $w_d$
with same properties as in Theorem \ref{th4}

\subsection{Numerical experiments}

In the illustrative numerical experiments described below we have
restricted to mixtures of univariate Gaussian variables $\mu_\a$
(hence with $\om=\R$ and $\a=({\rm mean,deviation})\in\A\subset\R^2$).
To implement the semidefinite relaxations \eqref{new-primal-lp-stokes}
we have used the GloptiPoly software \cite{gloptipoly} dedicated to solving the Generalized Moment Problem. The resulting SDPs are solved using version 8.1 of Mosek \cite{mosek}. 

We discuss three examples chosen to a) illustrate the effect (and efficiency) of Stokes constraints, b) compare the approximations with the real feasible set 
$\X^*_\varepsilon$ in \eqref{new-set} (approximated with intensive simulations), and c) show the behavior of the approximations for different violation probabilities. 

\subsection*{Approximations with and without Stokes}
In order to illustrate the difference in quality  of the approximation of $\X^*_\varepsilon$ when using or not using Stokes constraints, consider the example where $\X=[-1,1]$, $f(\x,\o)=\o-\x$, $\A = [-0.1,0.1]\times[0.8,1]$, i.e., we consider univariate Gaussian measures with mean approximately $0$ and deviation slightly less than $1$.
For every fixed $\x$, due to the simple expression of 
$f$ we can express ${\rm Prob}_{\mu_\a} (\{\o\in\om: g(\x,\o)<0\})$ as an analytic expression in $\a$. It is hence relatively easy to obtain a good estimation of 
$\rho(\x):=\min_{\a\in\A}{\rm Prob}_{\mu_\a} (\{\o\in\om: g(\x,\o)<0\})=1-\kappa(\x)$ 
by sampling over $\a$ and taking the minimum. In Figure \ref{fig: 1d} is displayed $\x\mapsto \rho(\x)$ in black and two different approximations $w_d$ computed for relaxation orders $d=8$ in blue and $d=12$ in red. The dashed lines are the polynomials corresponding to problem formulations including Stokes constraints.

\begin{figure}[ht]
\includegraphics[width=0.6\textwidth]{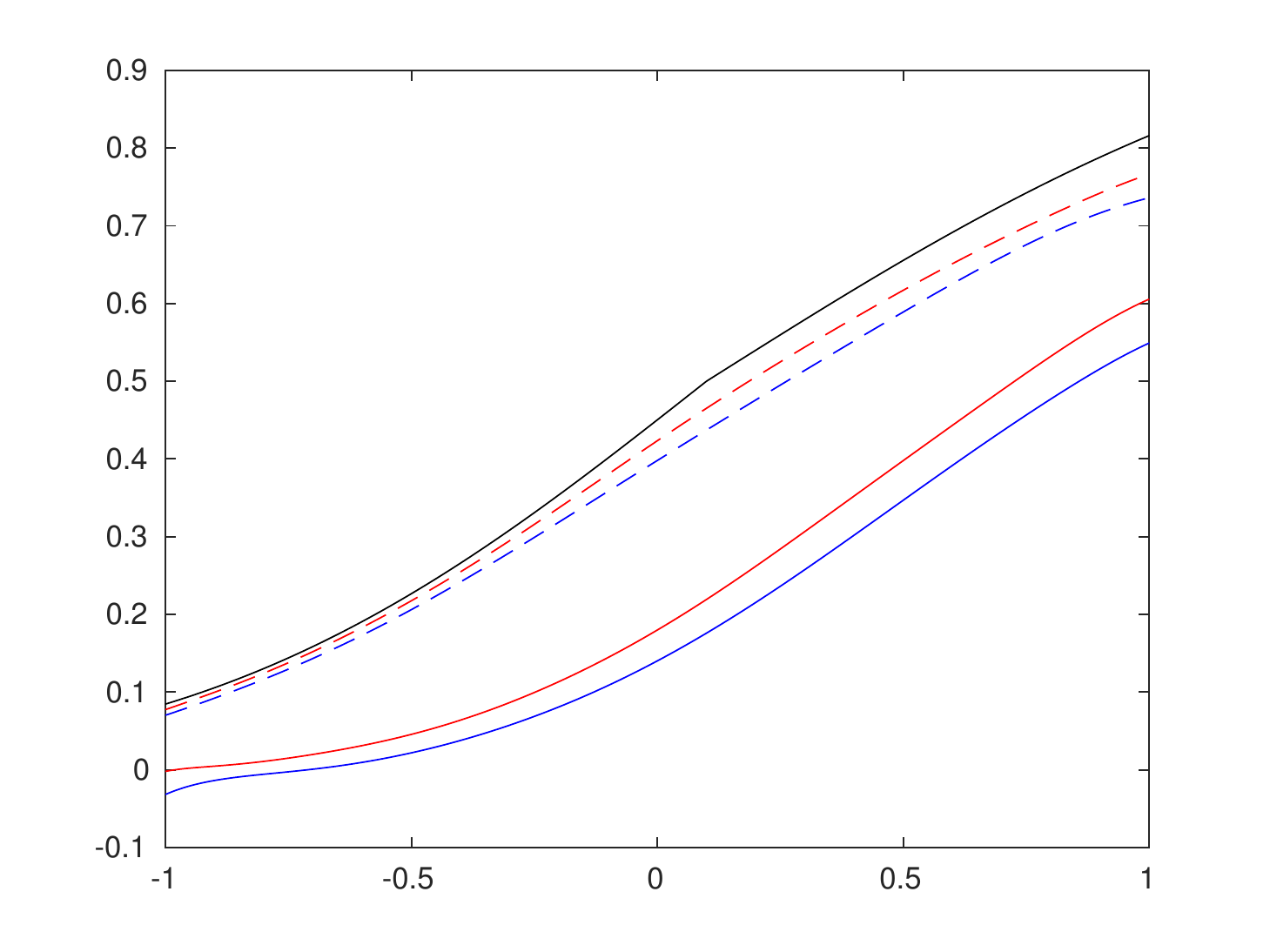}
\caption{Approximation of $\rho(\x)$ (black) by polynomials $1-w_8(\x)$ (blue) and $1-w_{12}(\x)$ (red), dashed/solid lines correspond to with/without Stokes constraints \label{fig: 1d}}
\end{figure}

As a first remark observe that, in accordance with the theoretic results, all approximations are underestimators of $\rho(\x)$. However, the approximations $(1-w_d)$ computed with Stokes constraints are much closer to $\rho$ than the ones computed without. The former approximations are particularly close to $\rho$ for significant values of violation probability, i.e., for small probabilities on the vertical axis. For higher probabilities they degrade (but are still quite good). This can be due to the non-differentiability of 
$\rho$ at $\x=0.1$. In order to display $\X^*_\varepsilon$ and 
its  $\X^d_\varepsilon:=\{\x\in\X: w_d(\x)\leq\varepsilon\}$, e.g., for a violation probability of 
$30\%$ ($\varepsilon=0.3$) one looks at the sets $\{\x\in\X: \rho(\x) \geq0.7\}$ 
and $\{\x\in\X: 1-w_d(\x)\geq0.7\}$ with $w_d$ an optimal solution of the dual for the
analogue of the step-$d$ relaxation  of \eqref{new-primal-lp-stokes}.
This yields approximately that the interval $[0.62,1]$ is the true feasible set. With Stokes, the approximations $w_{8}$ and $w_{12}$ yield the respective intervals $[0.85,1]$ and  $[0.75,1]$
while the approximations without Stokes provide an empty interval.

\subsection*{Inner approximations from various relaxations}
As seen in the previous example, Stokes constraints are essential for the performance of our approach. In this section we therefore only report results using these additional constraints. 
In the second illustrative example, $\X = [-1,1]^2$, $f(\x,\o)= 2\o\,\x_2^2 - 2\o\,\x_1^2 - 1$ and mean and deviation as in the example before in an environment of $0$ and $0.9$ respectively. In Figure \ref{fig: 2d} we plot the feasible set $\X^*_\varepsilon$ and its approximations $(\X^d_\varepsilon)$ for a violation level of $10\%$ ($\varepsilon=0.1$).

The feasible set is approximated as follows. We discretize $\X$  into $200$ and $\A$ into $100$ steps in each direction respectively. For each point $\x$ and each combination of parameters $\a$ we draw $1000$ realizations of $\o$ from the normal distribution described by $\a$. The point $\x$ is considered to be feasible whenever  for each $\a$, $f(\x,\o)$ is positive for at least $900$ out of the $1000$ realizations of $\o$. This simulation takes about $8600$ seconds (without the authors claiming to be experts for Monte Carlo simulations) whereas the approximations for $d = 8,10,12$ take $5$, $43$, and $482$ seconds respectively.

\begin{figure}[ht]
\includegraphics[width=0.6\textwidth]{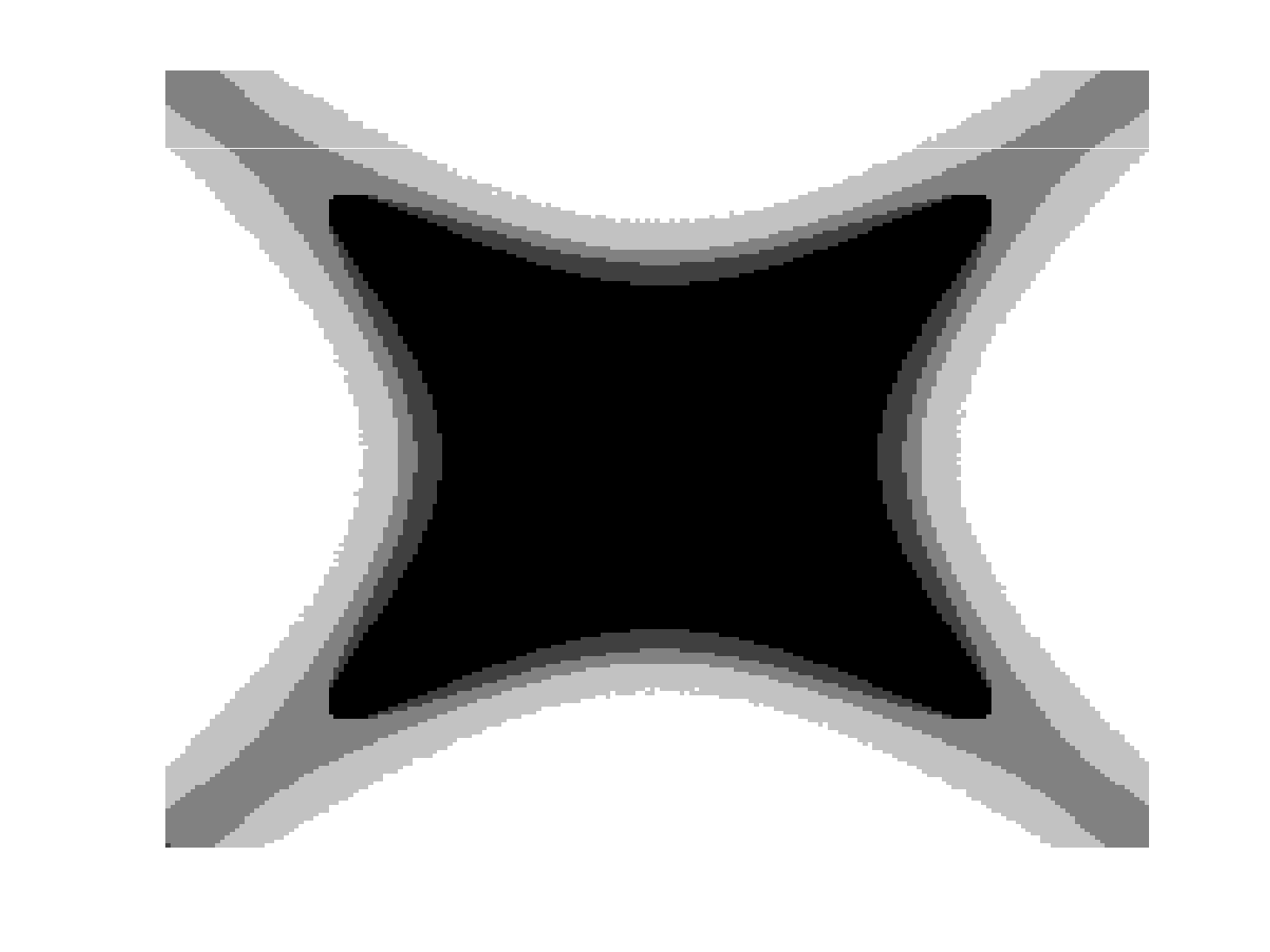}
\caption{Monte Carlo simulation (light grey) of $\X^*_\varepsilon$ and inner approximations 
$\X^d_\varepsilon$ for $d = 8,10,12$, in decreasing intensity \label{fig: 2d}}
\end{figure}

Inspection of Figure \ref{fig: 2d} reveals that the feasible set $\X^*_\varepsilon$  is non-convex. Already the lowest approximation $\X^8_\varepsilon$ (black) is able to capture this behavior. The next 
approximation $\X^{10}_\varepsilon$ (dark grey) is already a bit larger and
$\X^{12}_\varepsilon$ (medium grey) captures a significant part of $\X^*_\varepsilon$ ($\approx 74\%$). 
Its computation time is $18$ times faster than the the one required for the Monte Carlo simulation
of $\X^*_\varepsilon$. In addition, and in contrast to the approximation via Monte Carlo, 
$\X^d_\varepsilon$ is guaranteed to be inside the true feasible set.

\subsection*{Inner approximations with different violation levels}
In the third example, $\X=[-1,1]^3$, $\om=\R$, $f(\x,\o) = -2\o\,x_1^2+2\o\,x_2^2-2\o\,x_3^2-1$. 
We compute the inner approximations $\X^d_\varepsilon$ for  $d=8,10,12$.
To compute the Monte Carlo approximation of $\X^*_\varepsilon$ in a reasonable time, we fix the mean of the distribution to $0$ and  the standard deviation $\sigma$ is taken in the interval $\A=[0.4,0.6]$. For Monte Carlo we discretize $\X$ and $\A$ in $100$ steps in each direction and draw again $1000$ realizations of $\o$ for each point and each $\sigma$. This simulation takes about $2277$ seconds. In the first example we have already seen that the polynomial approximations $w_d$ are quite good for large violation probabilities. In Table \ref{tab: 3d} we compare the ``volume" of our approximations against the Monte Carlo simulation, i.e., the ratio of the number of points admissible for our approximations over the number of points admissible in Monte Carlo. As the polynomial approximations are inner approximations, we expect the ratii to be less than one 
(assuming that Monte Carlo is accurate). 

\begin{table}
\renewcommand{\arraystretch}{1.5}
\begin{tabular}{|l|rrrrr|}
\hline
$(r,\mbox{time})\backslash\varepsilon$	& $50\%$		&$25\%$		& $12.5\%$	& $6.25\%$	& $3.125\%$\\
\hline
$\hspace{5pt}8$ (\hspace{3pt}30s)		& $96.94\%$	& $83.07\%$	& $69.70\%$	& $22.72\%$	& $0\%$\\
$10$ (107s)	& $99.91\%$	& $86.70\%$	& $73.21\%$	& $73.79\%$& $2.48\%$\\
$12$ (633s)	& $100.0\%$	& $90.13\%$	& $79.94\%$	& $61.31\%$& $27.98\%$\\
\hline
\end{tabular}
\vspace{6pt}
\caption{Polynomial approximations vs Monte Carlo simulation.}
\label{tab: 3d}
\end{table}

Again the polynomial approximations $(w_d(x_1,x_2,x_3))$ are computed significantly faster than the Monte Carlo approximation $\rho(x_1,x_2,x_3)$. As in the first example, for large
$\varepsilon$ the approximations are pretty exact. However, 
for all relaxation orders $d$ the quality of approximation decreases with $\varepsilon$,
and eventually $\X^8_{0.03125}=\emptyset$.
However we should not forget that good approximations with small $\varepsilon$ are difficult to achieve in any case. Therefore it is quite interesting that we can retrieve almost $30\%$ of 
$\X^*_{0.03125}$ with $\X^{12}_{0.03125}$ and using moments up to order $12$ only.

\section{Extensions}

With some {\em ad-hoc} adjustments, the framework presented in this paper can be extended
to consider problems with:

 $\bullet$ only first- and second-order moments knowledge (no information about the distributions contributing to the mixture), and 
 
 $\bullet$ distributionally robust {\em joint} chance-constraints.
 
\subsection{Modeling with only first and second order moments}
\label{newsection}
As mentioned in Remark \ref{mu-sigma}, another possible and related ambiguity set is to consider the family $\mathscr{M}_\a$ of measures on $\om$ whose only first and second-order moments $\a=(\bm,\Sigma)$ belong to some prescribed set $\A$. 
The approach described in this paper also works with the following modifications.
 	
Suppose that $\o$ follows some unknown distribution on $\om\subset\R^p$ whose first and second order moments $\a=(\boldsymbol{m},\Sigma)\in \A$, e.g. 
$\A=[\underline{\boldsymbol{m}},\overline{\boldsymbol{m}}]\times \Theta$ where
$\Theta:=\{\Sigma:\underline{\delta}I\preceq \Sigma\preceq\overline{\delta}I\}$, for some $\underline{\delta}>0$. Then $\A$ is basic semi-algebraic set in $\R^{p(p+3)/2}$
defined by $4p$ polynomial inequalities. For instance, with 
\[\boldsymbol{m}=(m_1,m_2);\quad \Sigma=\left[\begin{array}{cc}\sigma_{11} &\sigma_{12}\\
\sigma_{12}&\sigma_{22}\end{array}\right],\]
\[\begin{array}{rl}\A= \{&\underline{m}_i\leq m_i\leq\overline{m}_i,\:i=1,2;\quad
2\underline{\delta} \leq\sigma_{11}+\sigma_{22}\leq 2\overline{\delta};\\
&(\overline{\delta}-\sigma_{11})(\overline{\delta}-\sigma_{22})-\sigma_{12}^2\geq0;\:
(\sigma_{11}-\underline{\delta})(\sigma_{22}-\underline{\delta})-\sigma_{12}^2\geq0\}.\end{array}\]
The infinite-dimensional LP \eqref{new-primal-lp} now becomes:
\[\begin{array}{rl}
\rho=\displaystyle\sup_{\phi,\nu,\mu,\psi}&\{\:\phi(\K):\:
\phi+\nu\,=\,\mu;\quad \mu_\x=\psi_\x=\lambda;\\
&\displaystyle
\int \x^\alpha\,\o_i\,d\phi+\int \x^\alpha\,\o_i\,d\nu=
\int \x^\alpha m_i\,d\psi,\quad \alpha\in\N^n,\:i=1,\ldots,p\\
&\displaystyle\int \x^\alpha\,\o_i\o_j\,d\phi+\int \x^\alpha\,\o_i\o_j\,d\nu=\int \x^\alpha \sigma_{ij}\,d\psi,\quad \alpha\in\N^n,\:1\leq i<j\leq p\\
&\displaystyle\phi\in\mathscr{\M}_+(\K),\psi,\mu\in\mathscr{\M}_+(\X\times\om),
\psi\in\mathscr{\M}_+(\X\times\A)\,\}.
\end{array}\]
Then semidefinite relaxations analogues of \eqref{chance-sdp} are defined in the obvious way
and their associated monotone sequence of optimal values $(\rho_d)_{d\in\N}$ converges to $\rho$ as $d$ increases. As for \eqref{chance-sdp-dual}, from an optimal solution of their dual one provides inner approximation of $(\X^d_\varepsilon)_{d\in\N}$ of $\X^*_\varepsilon$ and analogues of Theorem \ref{th4} and \ref{th3} also hold.

\subsection{Joint chance-constraints}

The case of {\em joint} chance-constraints, i.e., when several probabilistic constraints 
\begin{equation}
\label{joint-chance}
{\rm Prob}_\mu(f_j(\x,\o) >0,\:j=1,\ldots s_f)\,>\,1-\varepsilon,\end{equation}
are considered jointly, is in general significantly more complicated than its relaxation 
which considers them individually, i.e.,
\[{\rm Prob}_\mu(f_j(\x,\o) >0)\,>\,1-\varepsilon,\quad j=1,\ldots,s_f.\]
For instance, tractable formulations valid for individual chance-constraints
may not be valid any more for joint chance-constraints.\\

We next show that  joint chance-constraints (\ref{joint-chance}) can be modelled in our framework, relatively easily. Instead of the set $\K$ in (\ref{set-k}) we now consider the sets:
\begin{eqnarray}
\label{joint-set-k-j}
\K^j\,&:=&\{\,(\x,\o)\in \X\times\om:\:f_j(\x,\o)\,\leq\,0\,\},\quad j=1,\ldots,s_f.\\
\label{joint-set-kx-j}
\K_\x^j&:=&\{\,\o\in \om:\:(\x,\o)\,\in\,\K^j \,\},\quad j=1,\ldots,s_f;\quad\forall\x\in\X\\
\label{joint-set-k}
\K&:=&\{\,(\x,\o)\in \X\times\om:\:(\x,\o)\,\in\,\bigcup_{j=1}^{s_f}\K^j \,\},\\
\label{joint-set-kx}
\K_\x&:=&\{\,\o\in \om:\:(\x,\o)\,\in\,\bigcup_{j=1}^{s_f}\K_\x^j \,\},\quad\forall\x\in\X
\end{eqnarray}

All results of \S \ref{lp-formulation}, i.e., Theorem \ref{th1-lp} and Theorem \ref{th-abstract},
remain valid with now $\K$ and $\K_\x$, $\x\in\X$,
as in (\ref{joint-set-k}) and (\ref{joint-set-kx}) respectively. Indeed Lemma \ref{mu-star}
remains valid with $\K_\x$ as (\ref{joint-set-kx}). (In particular we still have $\mu_\a(\partial\K_\x)=0$ for all $\a\in\A$, as now the boundary 
$\partial\K_\x$ is contained in a finite union of zero sets of polynomials.)

What is not obvious is how to define the analogues of the semidefinite relaxations
(\ref{new-primal-lp}) because $\K$ is {\em not} a {\em basic} semi-algebraic set any more. It is a finite union 
$\bigcup_{j=1}^{s_f}\K^j$ of basic semi-algebraic sets with overlaps.\\

The analogue of the infinite-dimensional LP (\ref{new-primal-lp}) reads:
\begin{equation}
\label{joint-primal-lp}
\begin{array}{rl}
\hat{\rho}=\displaystyle\sup_{\phi_j,\psi\geq0}&
\{\,\displaystyle\sum_{j=1}^{s_f}\int_{\K^j}d\phi_j: \sum_{j=1}^{s_f}\phi^j\leq\T^*\psi;\quad\psi_\x\,=\,\lambda,\\
&\phi^j\in \mathscr{M}_+(\K^j),\:\psi\in\mathscr{P}(\X\times\A)\,\},\end{array}
\end{equation}
where $\T^*$ is defined in Lemma \ref{def-tstar}. It is important to emphasize that even though the sets $\K^j$ overlap, we do {\em not} require that the measures $\phi_j$ are mutually singular.

The dual of (\ref{joint-primal-lp}) reads:
\begin{equation}
\label{joint-lp-dual}
\begin{array}{rl}
\hat{\rho}^*\,=\,\displaystyle\inf_{h,w}&\{\displaystyle\int_\X w\,d\lambda:\quad h(\x,\o)\,\geq\,1\quad\mbox{on $\K^j$},\:j=1,\ldots,s_f\\
&w(\x)-\T h(\x,\a)\geq0\quad \mbox{on $\X\times\A$};\quad h\geq0\mbox{ on $\X\times\om$},\\
&\\&w\in\R[\x];\:h\in\R[\x,\o]\,\}.
\end{array}
\end{equation}

\begin{thm}
\label{th1-joint}
Let $\K^j$, $j=1,\ldots,s_f$, and $\K$ be as in (\ref{joint-set-k-j}) and (\ref{joint-set-k}) respectively. Then:

(i) The optimal value
$\rho$ of (\ref{new-primal-lp}) and the optimal value $\hat{\rho}$ of (\ref{joint-primal-lp}) are identical.

(ii) Define the functions:
\begin{equation}
\label{partition}
(\x,\o)\mapsto \theta_j(\x,\o)\,:=\,\frac{1}{\vert\{\,\ell\in\{1,\ldots,s_f\}: (\x,\o)\,\in\,\K_\ell\}\,\vert},\quad j=1,\ldots,s_f,\end{equation}
and let $(\phi^*,\psi^*)$ be an optimal solution of (\ref{joint-primal-lp}). Then
$(\phi^*_1,\ldots,\phi^*_{s_f},\psi^*)$ with
\begin{equation}
\label{joint-phi-star}
d\phi^*_j(\x,\o)\,:=\,1_{\K_j}(\x,\o)\,\theta_j(\x,\o)\,d\phi^*(\x,\o),\quad j=1,\ldots,s_f,\end{equation}
is an optimal solution of (\ref{joint-primal-lp}) and $\phi^*=\sum_j\phi^*_j$.
\end{thm}
\begin{proof}
Let $(\phi_1,\ldots,\phi_{s_f},\psi)$ be an arbitrary  feasible solution of (\ref{joint-primal-lp}) and let $\phi:=\sum_{j=1}^{s_f}\phi_j$. Then $\phi\in\mathscr{M}(\K)_+$ and 
$\phi\leq\T^*\psi$. Therefore $(\phi,\psi)$ is feasible for (\ref{new-primal-lp}). Morover,
as ${\rm supp}(\phi_j)\subset\K^j\subset\K$ for all $j=1,\ldots,s_f$,
\[\sum_{j=1}^{s_f}\int_{\K^j}d\phi_j\,=\,\sum_{j=1}^{s_f}\int_{\K}d\phi_j
\,=\,\int_\K\,d(\sum_{j=1}^{s_f}\phi_j)\,=\,\int_\K\,d\phi\,=\,\phi(\K),\]
which shows that $\hat{\rho}\leq\,\rho$. To prove the reverse inequality consider
the functions $(\theta_j)$ in (\ref{partition}),
and let $(\phi^*,\psi^*)\in\mathscr{M}_+(\K)\times\mathscr{M}_+(\X\times\A)$ be an optimal solution of (\ref{new-primal-lp}). 
Observe that 
\[\sum_{j=1}^{s_f}\theta_j(\x,\o)\,=\,1,\quad\forall (\x,\o)\in\K.\]
Let $(\phi^*_j)$, $j:=1,\ldots,s_f$, be as in (\ref{joint-phi-star}).
Then $\phi^*_j\in\mathscr{M}_+(\K_j)$, $j=1,\ldots,s_f$, and 
$\displaystyle\sum_{j=1}^{s_f}\phi^*_j(\K_j)=\phi^*(\K)$. Hence $\rho\leq\hat{\rho}$. Therefore
$(\phi^*_1,\ldots,\phi^*_{s_f},\psi^*)$ is an optimal solution of (\ref{joint-primal-lp}) and 
$\phi^*=\sum_j\phi^*_j$.
\end{proof}

As a consequence of Theorem \ref{th1-joint}, Theorem \ref{th-abstract} also holds
with now $\K_\x$ as in (\ref{joint-set-kx}). In the original proof just use 
$\phi=\sum_j\phi_j$ and the definitions
of $\K$ and $\K_\x$ in (\ref{joint-set-k})-(\ref{joint-set-kx}).

\subsection{Semidefinite relaxations}

We briefly describe the semidefinite relaxations of the LP (\ref{joint-primal-lp}), 
which are the analogues of (\ref{chance-sdp}) for the LP (\ref{new-primal-lp}).
For every $j=1,\ldots,s_f$ let $g^*_j:=-f_j$  and let $d^*_j:=\lceil{\rm deg}(g^*_j)/2\rceil$. Let $2\replaced{d_{\mathrm{min}}}{d_0}$ to be the largest 
degree appearing in the polynomials that describe $\K,\om,\A$, and
consider the semidefinite programs indexed by $d\geq \replaced{d_{\mathrm{min}}}{d_0}$.
\begin{equation}
\label{joint-sdp}
\begin{array}{rl}
\rho_d=&\displaystyle\sup_{\y^j,\u,\v}\sum_{j=1}^{s_f}y^j_{00}\\
\mbox{s.t.}& L_{\y^1+\cdots +\y^{s_f}+\u}(\x^\alpha\o^\beta)-L_\v(\x^\alpha p_\beta(\a))\,=0,\:\vert\alpha+{\rm deg}(p_\beta)\vert\leq 2d,\\
&L_\v(\x^\alpha)=\lambda_\alpha,\quad\alpha\in\N^n_{2d},\\
&\M_d(\y^j),\M_d(\u),\M_d(\v)\,\succeq\,0,\\
&\M_{d-d^*_j}(g^*_j\,\y^j)\succeq0,\quad j=1,\ldots,s_f,\\
&\M_{d-d_\ell}(g_\ell\,\y^j),\: \M_{d-d_\ell}(g_\ell\,\u),\,\M_d(g_\ell\,\v)\succeq0,\\
&\ell=1,\ldots,m;\:j=1,\ldots,s_f,\\
&\M_{d-d^1_\ell}(s_\ell\,\y^j),\:\M_{d-d^1_\ell}(s_\ell\,\u)\,\succeq 0,\\
&\ell=1,\ldots\bar{s};\:j=1,\ldots,s_f,\\
&\M_{d-d'_\ell}(q_\ell\,\v)\,\succeq0,\quad \ell=1,\ldots,L,
\end{array}
\end{equation}
where $\y^j =(y^j _{\alpha\beta})$, 
$\u=(v_{\alpha\beta})$, $(\alpha,\beta)\in\N^n\times\N^p$, $j=1,\ldots,s_f$, and
$\v=(v_{\alpha\eta})$, $(\alpha,\eta)\in\N^n\times\N^t$.\\

The dual of \eqref{joint-sdp} 
is a reinforcement of \eqref{joint-lp-dual} and its interpretation in terms of SOS positivity certificates 
of size parametrized by $d$ (the analogue of (\ref{chance-sdp-dual})) reads:

\begin{equation}
\label{joint-sdp-dual}
\begin{array}{rl}
\rho^*_d=\displaystyle\inf_{h,w,\sigma^i_j} &\displaystyle\int_\X w(\x)\,d\lambda(\x):\\
\mbox{s.t}&h(\x,\o)-1=\displaystyle\sum_{\ell=0}^{m}\sigma^1_{j\ell}\,g_\ell+\sigma^1_{j,m+1}g^*_j
+\displaystyle\sum_{\ell'=1}^{\bar{s}}\sigma^1_{j\ell'}\,s_{\ell'},\quad\forall (\x,\o);\\
&\mbox{ for all }j=1,\ldots,s_f;\\
&h(\x,\o)=\displaystyle\sum_{\ell=0}^m\sigma^2_\ell\,g_\ell
+\displaystyle\sum_{\ell'=1}^{\bar{s}}\sigma^2_{\ell'}\,s_{\ell'},\quad\forall (\x,\o);\\
&w(\x)-\displaystyle\sum_{\alpha,\beta}h_{\alpha\beta}\,\x^\alpha\,p_\beta(\a)=
\displaystyle\sum_{\ell=0}^m\sigma^3_\ell\,g_\ell+\displaystyle\sum_{\ell'=1}^L\sigma^3_{\ell'}\,q_{\ell'},\quad\forall (\x,\a);\\
&{\rm deg}(h),\:{\rm deg}(w)\leq 2d;\:\sigma_\ell^1\in\Sigma[\x,\o]_{d-d_\ell},j=0,\ldots,m\\
&\sigma^1_{j\ell}, \sigma^2_\ell\in\Sigma[\x,\o]_{d-d_\ell};\:\sigma^1_{j,m+1}\in\Sigma[\x,\o]_{d-d^*_j},\:\ell=0,\ldots,m;\:j=1,\ldots,s_f\\
&\sigma_\ell^3\in\Sigma[\x,\a]_{d-d_\ell};\ell=0,\ldots,m;\:\sigma_{\ell'}^3\in\Sigma[\x,\a]_{d-d'_\ell};\:\ell'=1,\ldots,L\\
&\sigma^1_{j\ell'},\sigma^2_{\ell'}\in\Sigma[\x,\o]_{d-d'_\ell},\:j=1,\ldots,s_f;\:\ell'=1,\ldots,L
\end{array}
\end{equation}
where $h(\x,\o)=\sum_{\vert\alpha+\beta\vert\leq 2d}h_{\alpha\beta}\,\x^\alpha\,\o^\beta$, and $w(\x)=\sum_{\vert\alpha\vert\leq 2d}w_{\alpha}\,\x^\alpha$.\\

For every $d\geq \replaced{d_{\mathrm{min}}}{d_0}$, $\rho_d\geq \hat{\rho}$ ($=\rho$) and Theorem \ref{th3} is valid for the hierarchy of semidefinite relaxations (\ref{joint-sdp}). Next, under the same assumptions
of non-empty interior for $\X,\A,\om,\K,\K^j$ and $(\X\times\om)\setminus \K$,
Theorem \ref{th4} is also valid for the dual hierarchy \eqref{joint-sdp-dual}.

\section{Conclusion}

Computing or even approximating the feasible set associated with a distributionally-robust chance-constraint is a challenging problem. We have described a systematic numerical scheme 
which provides a monotone sequence (a hierarchy) of inner approximations, all in the form 
$\{\x\in\X: w_d(\x)<\varepsilon\}$ for some polynomial of increasing degree $d$, with strong asymptotic guarantees as $d$ increases. 
To the best of our knowledge it is the first result of this type at this level of generality.
Of course this comes with a price as the polynomial 
which defines each approximation is obtained by solving a semidefinite program 
whose size increases with its degree. Therefore and so far, this approach 
is limited to problems of small dimension (except perhaps if
some sparsity can be exploited). So in its present form this contribution should be 
considered as complementary (rather than a competitor) to other algorithmic approaches where
scalability is of primary importance. However it may also provide useful insights 
and a benchmark (for small dimension problems) for the latter approaches.

\section{Appendix}

\begin{lem}
\label{mom-det}
Under Assumption \ref{ass-on-Ma-2}, every 
$\mu\in\mathscr{M}_\a$ is moment determinate.
\end{lem}
\begin{proof}
As $\mu\in\mathscr{M}_\a$, there exists $\varphi\in\mathscr{P}(\A)$ such that
\[\mu(B)\,=\,\int_\om\mu_\a(B)\,d\varphi(\a),\qquad B\in\mathcal{B}(\om).\]
By Assumption \ref{ass-on-Ma-2}, there exists $c,\gamma$ such that for every $i=1,\ldots,p$, 
\[\int_\om \exp(c\,\vert\o_i\vert)\,d\mu_\a(\o)
\,<\,\gamma,\qquad\forall\a\in\A,\]
and therefore 
\[\sup_{i=1,\ldots,p}\int_\om \exp(c\,\vert\o_i\vert)\,d\mu(\o)\,<\,\gamma.\]
As $(c\,\o_i)^{2j}/(2j){\rm !}<\exp(\vert\o_i\vert)$, one obtains
$\int_\om \o_i^{2j}\,d\mu(\o)<c^{-2j}\gamma\,(2j){\rm !}$ for all $j=1,,2,\ldots$ and all $i=1,\ldots,p$.
But this implies
\[\sum_{i=1}^\infty \left(\int_\om \o_i^{2j}\,d\mu(\o)\right)^{-1/2j}\,=\,+\infty,\qquad i=1,\ldots,p,\]
that is, $\mu$ satisfies Carleman's condition (\ref{lem-carleman}), and so is moment determinate.
\end{proof}

\subsection{Proof of Lemma \ref{mu-star}}
\label{proof-mu-star}
\begin{proof}
With $\x\in\X$ fixed, let $\theta:=\sup_{\a\in\A}\,\mu_\a(\K_\x)$.
For every $\mu\in\mathscr{M}_\a$, there exists $\varphi\in\mathscr{P}(\A)$ such that
$\mu(\K_\x)=\int_\A \mu_\a(\K_\x)\,d\varphi(\a)$ and 
so $\mu(\K_\x)\leq\theta$ for all $\mu\in\mathscr{M}_\a$. 
Conversely $\sup_{\mu\in\mathscr{M}_\a}\mu(\K_\x)\geq\theta$ because $\mu_\a\in\mathscr{M}_\a$ for all $\a\in\A$.
Next 
let $v:\X\times\A\to [0,1]$ be given by $v(\x,\a):=\mu_\a(\K_\x)$. By construction
$0\leq v\leq 1$ on $\X\times\A$ and if 
$v(\x,\cdot)$ is upper-semicontinous on $\A$ for every $\x\in\X$, then  by \cite[Proposition 4.4, p. 2018]{lass-siopt-2010}
there exists a {\em measurable selector} $\x\mapsto \a(\x)\in\A$, $\x\in\X$, such that
$v(\x,\a(\x))=\max\,\{v(\x,\a):\a\in\A\}$, that is, the desired result (\ref{mu-star-1}) holds.

So it remains to prove that $v(\x,\cdot)$ is upper-semicontinuous on $\A$ for every $\x\in\X$.
In fact we even prove that $v(\x,\cdot)$ is continuous on $\A$ for every $\x\in\X$.
So let $(a_n)_{n\in\N}\subset\A$ with $a_n\to\a\in\A$ as $n\to\infty$.
Let $q$ be an arbitrary  bounded continuous function on $\om$. Then by Assumption \ref{ass-on-Ma}(iv),
\[\lim_{n\to\infty}\int_\om q\,d\mu_{\a_{n}}(\o)\,=\,\lim_{n\to\infty}Q(\a_{n})\,=\,Q(\a)\,=\,
\int_\om q\,d\mu_{\a}(\o),\]
which proves that $\mu_{\a_n}\Rightarrow\mu_{\a}$ as $n\to\infty$
(where $\Rightarrow$ denotes the weak convergence of probability measures ; see Billingsley \cite{billingsley}).
In addition, in view of the definition of $\K_\x$ in (\ref{set-kx}), its boundary $\partial\K_\x$  is contained in the zero set
of some polynomials and therefore, by Assumption \ref{ass-on-Ma}(iii), 
$\mu_{\a_n}(\partial\K_\x)=\mu_\x(\partial\K_\x)=0$ for all $n$ (i.e. $\partial\K_\x$ is
a $\mu_{\a_n}$-continuity set). Hence by the Portmanteau theorem \cite[Theorem 2.1, p. 11]{billingsley} it follows that
\[\lim_{n\to\infty}\mu_{\a_n}(\K_\x)\,=\,\lim_{n\to\infty}v(\x,\a_{n})\,=\,\mu_{\a}(\K_\x)\,=\,v(\x,\a),\]
i.e., $v(\x,\cdot)$ is continuous on $\A$ for every $\x\in\X$. In addition
$\kappa(\x)=v(\x,\a(\x))$ is also measurable.
\end{proof}

\subsection{Proof of Theorem \ref{th-abstract}}
\label{proof-th-abstract}

\begin{proof}
{\em Weak duality} holds 
because for every feasible solution $(w,h)$ of (\ref{chance-lp-dual}) and
$(\phi,\psi)$ of (\ref{new-primal-lp}), one has:
\begin{eqnarray*}
\int_\X w \,d\lambda=\int_{\X\times\A} w(\x)\,d\psi(\x,\a)&\geq&\int_{\X\times\A}\T h(\x,\a))\,d\psi(\x,\a)\\
&=&\int_{\X\times\om}h(\x,\o)\,\T^*\psi(d(\x,\o))\\
&\geq&\int_{\K}h(\x,\o)\,d\phi(\x,\o)\\
&\geq& \int_\K d\phi=\phi(\K).
\end{eqnarray*}
Moreover let $\psi^*$ be an optimal solution of (\ref{new-primal-lp}) as in Theorem \ref{th1-lp}, so
that $\psi^*=\delta_{\a(\x)}\lambda(d\x)$. Then for every $\x\in\X$:
\begin{eqnarray*}
w(\x)\,\geq\,\int_\A \T h(\x,\a)\,\psi^*(d\a\vert\x)&=&\int_\om h(\x,\o)\,d\mu_{\a(\x)}(\o)\\
&\geq&\int_{\K_\x} h(\x,\o)\,d\mu_{\a(\x)}(\o)\\
&\geq&\int_{\K_\x} d\mu_{\a(\x)}(\o)\,=\,\kappa(\x),
\end{eqnarray*}
i.e., $w(\x)\geq\kappa(\x)$ for all $\x\in\X$. In particular
\[\{\x\in\X: w(\x)<\varepsilon\}\,\subset\,\{\x\in\X:\kappa(\x)<\varepsilon\}\,=\,\X^*_\varepsilon.\]
Next, if there is no duality gap, i.e., if $\rho=\rho^*$,
then for a minimizing sequence $(w_n,h_n)$ of
(\ref{chance-lp-dual}), 
\[\lim_{n\to\infty}\int_\X(w_n(\x)-\kappa(\x))\,d\lambda(\x)\,=\,\rho^*-\rho\,=\,0,\]
that is, $w_n$ converges to $\kappa(\x)$ in $L_1(\X,\lambda)$.
Finally, by Ash \cite[Theorem 2.5.1]{Ash},  
convergence in $L_1(\X,\lambda)$ implies convergence in $\lambda$-measure, and so,
for every fixed $0<\ell\in\N$,
\begin{equation}
\label{conv-measure}
\lim_{n\to\infty}\lambda\,\{\x\in\X: \vert w_n(\x)-\kappa(\x)\vert>1/\ell\}\,=\,0.\end{equation}
Next, observe that 
\[\X^*_\varepsilon\,=\,\{\x\in\X:\kappa(\x)<\varepsilon\}\,=\,\bigcup_{\ell=1}^\infty \underbrace{\{\x\in\X:\kappa(\x)<\varepsilon-1/\ell\}}_{:=R_\ell}\]
and so $\lambda(\X^*_\varepsilon)\,=\,\lim_{\ell\to\infty}\lambda(R_\ell)$. Next
\[\lambda(R_\ell)\,=\,\lambda(R_\ell\cap\{\x: w_n(\x)<\varepsilon\})+\lambda(R_\ell\cap\{\x: w_n(\x)\geq\varepsilon\}).\]
By convergence in measure (\ref{conv-measure}), $\lim_{n\to\infty}\lambda(R_\ell\cap\{\x: w_n(\x)\geq\varepsilon\})=0$. Hence
\[\lambda(R_\ell)\,=\,\lim_{n\to\infty}\lambda(R_\ell\cap\{\x: w_n(\x)<\varepsilon\})\,\leq\,
\lim_{n\to\infty}\lambda(\{\x: w_n(\x)<\varepsilon\})\leq\lambda(\X^*_\varepsilon),\]
and as $\lambda(R_\ell)\to\lambda(\X^*_\varepsilon)$, $\lim_{n\to\infty}\lambda(\{\x: w_n(\x)<\varepsilon\})=\lambda(\X^*_\varepsilon)$.
\end{proof}

\subsection{Proof of Lemma \ref{lem-handle}}
\label{proof-lem-handle}

\begin{proof}
If $\om$ is compact then it follows from the definition of 
$\T$ and $\T^*$. For the general case where $\om$ is not necessarily compact, disintegrate $\nu$ and $\psi$ as
\[d\nu(\x,\om)\,=\,\hat{\nu}(d\o\vert\x)\,\nu_\x(d\x),\quad
d\psi(\x,\a)\,=\,\hat{\psi}(d\a\vert\x)\,\psi_\x(d\x).\]
By (\ref{moments}) with $\beta=0$,
$\int_\X\x^\alpha\nu_\x(d\x)=\int_\X\x^\alpha \psi_\x(d\x)$ for all $\alpha\in\N^n$,
and as $\X$ is compact it follows that $\psi_\x=\nu_\x$. Next,
fix $\beta\in\N^p$. Then for every $\alpha\in\N^n$
\[\int_\X \x^\alpha\,\left(\int_\om \o^\beta\,\hat{\nu}(d\o\vert\x)\right)\,\nu_\x(d\x)
\,=\,\int_\X \x^\alpha\,\left(\int_\A p_\beta(\a) \hat{\psi}(d\a\vert \x)\right)\,\nu_\x(d\x),\]
and again as $\X$ is compact this implies
\[\int_\om \o^\beta\,\hat{\nu}(d\o\vert\x)\,=\,
\int_\A p_\beta(\a) \hat{\psi}(d\a\vert \x),\quad\forall \x\in \X\setminus B_\beta,\]
where $B_\beta\in\mathcal{B}(\X)$ is such that $\nu_\x(B_\beta)=0$.
As $\beta\in\N^p$ was arbitrary,
\[\int_\om \o^\beta\,\hat{\nu}(d\o\vert\x)\,=\,
\int_\A p_\beta(\a) \hat{\psi}(d\a\vert \x),\quad\forall  \beta\in\N^p,\,\forall\x\in \X\setminus B^*,\]
where $\nu_\x(B^*)=\nu_\x(\bigcup_\beta B_\beta)=0$.
Next, define the measure $\varphi_\x$ on $\om$ by
\[\varphi_\x(B)\,=\,\int_\A \mu_\a(B)\,\hat{\psi}(d\a\,\vert\x),\qquad\forall
 \x\in \X\setminus B^*,\]
which is well defined by Assumption \ref{ass-on-Ma}(i). 
Moreover by construction, $\varphi_\x\in\mathscr{M}_\a$  for all
$\x\in \X\setminus B^*$, and
\begin{eqnarray*}
\int_\om\o^\beta\,\hat{\nu}(d\o\vert\x)&=&
\int_\A p_\beta(\a)\,\hat{\psi}(d\a\vert\x),\quad \forall\beta\\
&=&\int_\A \left(\int_\om \o^\beta \mu_\a(d\o)\right)\,\hat{\psi}(d\a\vert\x),\quad\forall\beta\\
&=&\int_\om\o^\beta \,\varphi_\x(d\o),\quad\forall\beta.
\end{eqnarray*}
By Lemma \ref{mom-det}, $\varphi_\x$ is moment determinate and therefore $\hat{\nu}(d\o\vert\x)=\varphi_\x$ for all $\x\in\X\setminus B^*$.
Next, let $g\in\mathscr{B}(\X\times\om)$ be fixed arbitrary. Then:
\begin{eqnarray*}
\langle g,\nu\rangle \,=\,\langle g,\hat{\nu}(d\o\vert\x)\,\nu_\x(d\x)\rangle&=&
\langle g,\varphi_\x(d\o)\nu_\x(d\x)\rangle\\
&=&\langle g,\varphi_\x(d\o)\psi_\x(d\x)\rangle\\
&=&\langle \T g,\hat{\psi}(d\a\vert\x)\psi_\x(d\x)\rangle\\
&=&\langle \T g,\psi\rangle\,=\,\langle g,\T^*\psi\rangle,
\end{eqnarray*}
and as it holds for all $g\in\mathscr{B}(\X\times\om)$, $\nu=\T^*\psi$.
\end{proof}

\subsection{Proof of Theorem \ref{th4}}
\label{proof-th4}
\begin{proof}
(i) We first prove that Slater's condition holds for (\ref{chance-sdp}).
Observe that for all feasible solutions, $y_{00}+u_{00}=L_\v(1)=1$
and therefore $0\leq\rho_d\leq1$.
Let $\lambda_\A$ be the Lebesgue measure on $\A$, normalized to
a probability measure. 
Let $\y$ be the the moments of the measure $d\phi(\x,\o)=1_{\bS}(\x)\,d\lambda(\x)\otimes\hat{\mu}(d\o)$, where
\[\hat{\mu}(B)\,=\,\int_\A \mu_\a(B)\,d\lambda_\A(\a),\qquad B\in\mathcal{B}(\om).\]
Similarly, let $\nu:=\lambda\otimes\hat{\mu}$
(and so $\phi\leq\nu$) and let $\psi=\lambda_\A\otimes \lambda$.
Let $\y$ (resp. $\u$) be the vector of moments of $\phi$
(resp. $\nu-\phi$) up to order $2d$, and let $\v$ be the vector of moments of $\psi$ up to order $2d$.  Then $\M_d(\y),\M_d(\u),\M_d(\v)\succ0$.
Similarly $\M_{d-d_j}(\y)\succ0$, $j=0,\ldots,m$,
$\M_{d-d_j}(g_j\,\u)\succ0$, $j=1,\ldots,m$, and
$\M_{d-d'_\ell}(q_\ell\,\v)\succ0$, $\ell=1,\ldots,L$, because
$\X\times\om,\K,\X\times\om\setminus\K,\A$ all have nonempty interior.
Moreover, as
\[\int \x^\alpha\o^\beta\,d\nu(\x,\o)\,=\,\int_\X\x^\alpha\,d\lambda(\x)\,\int_\A(\int_\om\o^\beta\,d\mu_\a(\o))\,d\lambda_\A(\a))\]
\[=\int_\X\x^\alpha\,d\lambda(\x)\,\int_\A p_\beta(\a)\,d\lambda_\A(\a)
=\int_{\X\times\A}\x^\alpha\,p_\beta(\a)\,d\psi(\x,\a),\]
we deduce $L_{\y+\u}(\x^\alpha\o^\beta)=L_\v(\x^\alpha\,p_\beta(\a))$,
and therefore $(\y,\u,\v)$ is  an admissible solution of (\ref{chance-sdp}) which is strictly feasible, i.e., Slater's condition holds for (\ref{chance-sdp})
and therefore strong duality $\rho_d=\rho^*_d$ holds. In particular, as $\rho_d<\infty$,
(\ref{chance-sdp-dual}) has an optimal solution $(h_d,w_d,\sigma^i_j)$.

(ii) Next feasibility in (\ref{chance-sdp-dual}) implies
\[h_d(\x,\o)\,\geq\,1, \quad\forall (\x,\o)\in\K,\]
\[h_d(\x,\o)\,\geq\,0, \quad\forall (\x,\o)\in\X\times\om,\]
\[w_d(\x)-\sum_{\alpha,\beta}h_{d,\alpha,\beta}\,\x^\alpha\,p_\beta(\a)\,\geq\,0, \quad\forall (\x,\a)\in\X\times\A.\]
Let $(\phi^*,\psi^*)$ be optimal solution of (\ref{new-primal-lp}), as in Theorem \ref{th1-lp},
and let $d\nu^*(\x,\o):\mu_{\a(\x)}(d\o)\,\lambda(d\x)$. Let $\x\in\X$ be fixed.
Integrating the first w.r.t. $\nu^*_\x(d\o\vert\x)$ ($=\mu_{\a(\x)}(d\o)$), the third one w.r.t. $\psi^*(d\a\vert\x)$ ($=\delta_{\a(\x)}$), and using the second inequality  yields:
\[w_d(\x)+
\sum_{\alpha,\beta}h_{d,\alpha,\beta}\,\x^\alpha\,(\int_\om\o^\beta \,d\mu_{\a(\x)}(\o)-\,\underbrace{\int_\A p_\beta(\a)\,d\psi^*(d\a\vert\x)}_{=\int_\om\o^\beta\,d\mu_{\a(\x)}(\o)})\]
\[\geq\,\left\{\begin{array}{l}\displaystyle\int_{\K_\x}\nu^*_\x(d\o\vert\x)=\kappa(\x),\quad \mbox{if $\K_\x\neq\emptyset$,}\\
0\,(=\kappa(\x))\mbox{ otherwise.}\end{array}\right..\]
In other words 
\begin{equation}
\label{aux}
w_d(\x)\,\geq\,\kappa(\x),\quad\forall \x\in\X.
\end{equation}
Therefore
\[\{\x\in\X: w_d(\x) \,<\varepsilon\}\,\subseteq\,\{\x\in\X: \kappa(\x)\,<\varepsilon\}\,=\,\X^*_\varepsilon.\]
Next if $\lim_{d\to\infty}\rho_d=\rho$ then
\[\lim_{d\to\infty}\int_\X w_d\,d\lambda=\lim_{d\to\infty}\rho^*_d=\lim_{d\to\infty}\rho_d=\int_\X \kappa(\x)\,\lambda(d\x),\] which yields
\[\int_\X (w_d(\x)-\kappa(\x))\,d\lambda(\x)\,\to\,0\quad\mbox{as $d\to\infty$},\]
which combined with (\ref{aux}), yields $w_d\to\kappa$ in $L_1(\X,\lambda)$. 
\end{proof}

\subsection{Proof of Theorem \ref{th3}}
\label{proof-th3}
\begin{proof}
We prove Theorem \ref{th3} for the case where $\om$ is unbounded
as the arguments also work for the bounded case (but without Assumption \ref{ass-on-Ma} and 
\ref{ass-on-Ma-2}).

Let $(\y^k,\u^k,\v^k)$ be a maximizing sequence of (\ref{chance-sdp}).
For every $i=1,\ldots,p$ and $j\in\N$, let $\beta(i,j)\in\N^p$
be such that $\beta(i,j)_k=2j\delta_{k=i}$.
Observe that for every feasible solution $(\y,\u,\v)$ of (\ref{chance-sdp}), and from a
consequence of Assumption \ref{ass-on-Ma-2} (see the proof of Lemma \ref{mom-det}):
\begin{eqnarray}
\nonumber
L_\y(\o^{2j}_i)+L_\u(\o_i^{2j})&=& L_\v(p_{\beta(i,j)}(\a))\\
\label{bound}
&=&\int_\A p_{\beta(i,j)}(\a)\,d\lambda_\A(\a)\,\leq\,c^{-2j}\,\gamma\,(2j){\rm !}\,
\end{eqnarray}
for all  $j=1,\ldots,d$ and all $i=1,\ldots,p$. Therefore
$L_\y(\o^{2j}_i)\leq c^{-2j}\,\gamma\,(2j){\rm !}$ and
$L_\u(\o^{2j}_i)\leq c^{-2j}\,\gamma\,(2j){\rm !}$
for all $j=1,\ldots,d$ and all $i=1,\ldots,n$.
As $\A$ is compact and $q_1(\a)=M-\Vert\a\Vert^2$,
it follows that $L_\y(a_i^{2j})\leq M^j$, $j=1,\ldots d$, $i=1,\ldots,L$. By the same argument
using now $g_1(\x)=M-\Vert\x\Vert^2$, 
$L_\y(x_i^{2j})\leq M^j$, $L_\u(x_i^{2j})\leq M^j$,  and $L_\v(x_i^{2j})\leq M^j$,
$j=1,\ldots d$, $i=1,\ldots,L$. 

Next, as $y_{00}\leq1$, $\u_{00}\leq1$ and $v_{00}=1$, and $\M_d(\y),\M_d(\u),\M_d(\v)\succeq0$,
by invoking \cite[]{lass-netzer}, we obtain
\[\vert y_{\alpha\beta}\vert\,\leq\,\max[1,\max_i[L_\y(x_i^{2d}),L_\y(\o_i^{2d})]]\,=\,\tau_1,\]
\[\vert u_{\alpha\beta}\vert\,\leq\,\max[1,\max_i[L_\y(x_i^{2d}),L_\u(\o_i^{2d})]]\,=\,\tau_2;\quad
\vert v_{\alpha\beta}\vert\,\leq\,\max[1,M^d]\,=\,\tau_3,\]
for all $(\alpha,\beta)$. This implies that the feasible set of (\ref{chance-sdp}) is compact and 
so (\ref{chance-sdp})  has an optimal solution $(\y^d,\u^d,\v^d)$ for every $d\geq \replaced{d_{\mathrm{min}}}{d_0}$.

For every $d$ let $\tau_d:=\max[\tau_1,\tau_2,\tau_3]$. Next, by completing with zeros, consider the finite vectors  $\y^d$, $\u^d$ and $\v^d$ 
has infinite sequences. As $\vert y_{\alpha\beta}\vert\leq\tau_d$,
$\vert u_{\alpha\beta}\vert\leq\tau_d$ and $\vert v_{\alpha\eta}\vert\leq\tau_d$, whenever
$\vert\alpha+\beta\vert\leq 2d$ and $\vert\alpha+\eta\vert\leq 2d$,
by a standard argument
\footnote{Let $(u^n)_{n\in\N}$ be a sequence of infinite sequences such that
$\sup_n\vert u^n_i\vert<\tau_i$ for all $i=1,\ldots$, and let
$\hat{u}^n_i:=u^n_i/\tau_i$, for all $i,n$. Then $(\hat{u}^n)\subset\B_1$ where $\B_1$ is the unit ball of $\ell_\infty$. By weak-star sequential compactness
of $\B_1$, there is a subsequence $(n_k)$ and $\hat{u}\in\B_1$ such that
$\hat{u}^{n_k}\to \hat{u}$ 
for the $\sigma(\ell_\infty,\ell_1)$
weak-star topology of $\ell_\infty$. In particular, 
$\hat{u}^{n_k}_i\to\hat{u}_i$, as $k\to\infty$, for all $i=1,\ldots$, which implies
$u^{n_k}_i\to\tau_i\hat{u}_i$, as $k\to\infty$, for all $i=1,\ldots$.}
there exists
a subsequence $(d_k)_{k\in\N}$ and 
infinite sequences $\y^*=(y^*_{\alpha\beta})$, $\u^*=(u^*_{\alpha\beta})$,
and $\v^*=(v_{\alpha\eta})$, $\alpha\in\N^n$, $\beta\in\N^p$, and $\eta\in\N^L$, such that
\begin{equation}
\label{aa-1}
\lim_{k\to\infty} y^{d_k}_{\alpha\beta}\,=\,y^*_{\alpha\beta},\quad
\lim_{k\to\infty} u^{d_k}_{\alpha\beta}\,=\,u^*_{\alpha\beta},\quad\alpha\in\N^n,\,\beta\in\N^p,
\end{equation}
\begin{equation}
\label{aa-2}
\lim_{k\to\infty} v^{d_k}_{\alpha\eta}\,=\,v^*_{\alpha\eta},\quad\alpha\in\N^n,\,\eta\in\N^L.
\end{equation}
Next, fix $r\in\N$ arbitrary. By (\ref{aa-1}) and (\ref{aa-2}),
$\M_r(\y^*)\succeq0$, $\M_r(\u^*)\succeq0$ and $\M_r(\v^*)\succeq0$. In addition:
\[L_{\y^*}(\o^{2j}_i)^{-1/2j}\,\geq\,c\,\gamma^{-1/2j}\,(2j{\rm !})^{-1/2j},\quad \forall i,j,\]
and similarly 
\[L_{\u^*}(\o^{2j}_i)^{-1/2j}\,\geq\,c\,\gamma^{-1/2j}\,(2j{\rm !})^{-1/2j},\quad \forall i,j.\]
Also
\[L_{\y^*}(x^{2j}_i)^{-1/2j},\: L_{\u^*}(x^{2j}_i)^{-1/2j},\: L_{\v^*}(x^{2j}_i)^{-1/2j}
\,\geq\, M^{-1/2}\quad \forall i,j.\]
Therefore as $\sum_{j=1}^\infty (2j{\rm!})^{-1/2j}=+\infty$, one obtains
\begin{equation}
\label{crucial}
\sum_{j=1}^\infty L_{\y^*}(\o^{2j}_i)^{-1/2j}\,=\,+\infty,\quad\mbox{and}\quad
\sum_{j=1}^\infty L_{\y^*}(x_i^{2j})^{-1/2j}\,=\,+\infty,
\quad \forall i,j,\end{equation}
and similarly for $\u^*,\v^*$. In summary,
the three sequences $\y^*,\u^*$ and $\v^*$ satisfy the multivariate Carleman's condition
(\ref{carleman}). As $\M_d(\y^*),\,\M_d(\u^*),\,\M_d(\v^*)\succeq0$, 
they have a representing measure $\phi^*,\varphi^*$ and $\psi^*$ 
respectively on $\R^n\times\R^p$, $\R^n\times\R^p$ and $\R^n\times\R^t$.\\

Next, as (\ref{archimedian}) holds, the quadratic module of $\R[\x,\a]$
generated by the polynomials $(g_j,q_\ell)$ is Archimedean. Therefore, as 
$\M_d(g_j\,\v^*)\succeq0$ and $\M_d(q_\ell\,\v^*)\succeq0$ for all $d$ and
all $j=1,\ldots,m$, $\ell=1,\ldots,L$, by Putinar's Theorem \cite{putinar}, the measure $\psi^*$ is supported on $\X\times\A$. 
Also, as (\ref{archimedian}) holds, the quadratic module of $\R[\x]$
generated by the polynomials $(g_j)$ is Archimedean. Hence the 
marginal $\phi^*_\x$ of $\phi$ is supported on $\X$.
If $\om$ is compact (and as then $s_1(\o)=M-\Vert\o\Vert^2$) a similar argument shows that
$\phi^*$ is supported on $\K$.

If $\om$ is not compact, then by (\ref{bound}) and (\ref{aa-1})
we have $L_{\y^*}(\o_i^{2j})\leq c^{-2j}\,\gamma\,(2j{\rm !})$ for all $i,j$.
As $\M_d(s_\ell\,\y^*)\succeq0$ for all $\ell$, and $\M_d(-f\,\y^*)\succeq0$ 
for all $d$, then by \cite[Theorem 2.2, p. 2494]{lass-tams}\footnote{In the proof 
of Theorem 2.2 in \cite{lass-tams}, $c=1$, but the proof can be extended easily to
arbitrary $c>0$},
$f(\x,\o)\leq0$ and $s_\ell(\o)\geq0$, $\ell=1,\ldots,\bar{s}$, 
on the support of $\phi^*$. That is, ${\rm support}(\phi^*)\subset\K$.

Hence $(\phi^*,\psi^*)$
is a feasible solution of (\ref{new-primal-lp}).
But as $\phi^*(\K)=\lim_{d\to\infty}\rho_d\geq \rho$, we conclude that $(\phi^*,\psi^*)$
is an optimal solution with value $\rho$. 
\end{proof}

\subsection{Verifying Assumption \ref{ass-on-Ma-2} and Assumption \ref{ass-on-Ma-2}}
\label{verifying}

\subsubsection{$\A$ is a finite set}
 In this case $\om=\R^p$ and $\A=\{1,\ldots,\kappa\}$. Then
Assumption \ref{ass-on-Ma} holds and Assumption
\ref{ass-on-Ma-2} holds whenever it holds for
each individual $\mu_{i}$, $i=1,\ldots,\kappa$. The set $\mathscr{M}_\a$ can be identified with
the simplex $\Delta=\{\lambda\in\R^\kappa:\sum_i\lambda=1,\:\lambda\geq0\}$. In Theorem \ref{th1-lp}, 
the conditional probability $\hat{\phi^*}(d\a\vert\x)$, $\x\in\bS$, of the optimal solution $\phi^*$,
identifies the worst-case distribution $\mu_{\a(\x)}\in\A$ for every $\x\in\bS$.

\subsubsection{Mixture of Multivariate Gaussian distributions}

In the general case $\om=\R^p$, $\a=(\theta,\Sigma)$ with $\underline{\theta}\leq\theta\leq\overline{\theta}$, where
$\underline{\theta},\overline{\theta}\in\R^p$, and 
$\Sigma=\Sigma^T=(\sigma_{ij})\in\R^{p\times p}$, with $\underline{\delta}\,\mathbf{I}\preceq \Sigma\preceq \bar{\delta}\,\mathbf{I}$ and $\underline{\delta}>0$. That is,
\[d\mu_\a(\o)\,=\,\frac{1}{\sqrt{(2\pi)^p\,{\rm det}(\Sigma)}}\,\exp(-\frac{1}{2}(\o-\theta)\Sigma^{-1}(\o-\theta))\,d\o,\]

The measurability condition in Assumption \ref{ass-on-Ma}(i) follows from Fubini-Tonelli's theorem.
Assumption \ref{ass-on-Ma}(ii) is also satisfied. For instance, the fourth-order central moments read
(see e.g. {\tiny\verb+https://en.wikipedia.org/wiki/Multivariate_normal_distribution#Higher_moments+}):
\[\int (\o_i-\theta_i)^4\,d\mu_\a(\o)\,=\,3\,\sigma_{ii}^2;\:
\int (\o_i-\theta_i)^3(\o_j-\theta_j)\,d\mu_\a(\o)\,=\,3\,\sigma_{ii}\,\sigma_{ij}\]
\[\int (\o_i-\theta_i)^2(\o_j-\theta_j)^2\,d\mu_\a(\o)\,=\,\sigma_{ii}\,\sigma_{jj}+2\,\sigma_{ij}^2,\]
\[\int (\o_i-\theta_i)^2(\o_j-\theta_j)(\o_k-\theta_k)\,d\mu_\a(\o)\,=\,\sigma_{ii}\,\sigma_{jk}+2\,\sigma_{ij}\,\sigma_{ik},\]
\[\int (\o_i-\theta_i)(\o_j-\theta_j)(\o_k-\theta_k)(\o_\ell-\theta_\ell)\,d\mu_\a(\o)\,=\,\sigma_{ij}\,\sigma_{k\ell}+\sigma_{ik}\,\sigma_{j\ell}+
\sigma_{i\ell}\,\sigma_{jk},\]
and higher-order central moments are homogeneous polynomials in the entries of $\Sigma$. 
This immediately implies that non-central moments are polynomials in $(\sigma_{ij})$ and $\theta$. Assumption \ref{ass-on-Ma}(iii) is also straightforward as $\mu_\a$ has a density w.r.t. $d\o$, everywhere positive. Concerning Assumption \ref{ass-on-Ma}(iv), let $h$ be bounded continuous on $\X\times\om$.
With the change of variable $\u=\Sigma^{-1/2}(\o-\theta)$ one has
\[H(\x,\a):=\int_{\om} h(\x,\o)\,d\mu_\a(\o)\,=\,\frac{1}{\sqrt{(2\pi)^p}}\int_\om h(\x,\Sigma^{1/2}\,\u+\theta)\,\exp(-\frac{1}{2}\Vert\u\Vert^2)\,d\u,\]
and since $h$ is bounded and continuous, it follows that $H$ is continuous in $(\x,\a)\in\X\times\A$. Finally, Assumption \ref{ass-on-Ma-2} also holds.

\subsubsection{Mixture of exponential distributions}

In this case $\om=\R^p_+$, $\a=(a_1,\ldots,a_p)$ with $\a\in\A:=[\underline{\a},\overline{\a}]$, $\underline{\a}>0$, and
\[d\mu_\a(\o)\,=\,
\left(\prod_{i=1}^p \frac{1}{a_i}\right)\,\exp(-\displaystyle\sum_{i=1}^p\frac{\o_i}{a_i})\,d\o\,
=\,\otimes_{i=1}^p\,d\mu_{a_i}(\o_i)\]
with $d\mu_{a_i}(\o_i)=\frac{1}{a_i}\exp(-\o_i/a_i)\,d\o_i$, $i=1,\ldots,p$.

Again, the measurability condition in Assumption \ref{ass-on-Ma}(i) follows from Fubini-Tonelli's theorem.
Then for Assumption \ref{ass-on-Ma}(ii),
\[\int_\om \o^\beta\,d\mu_\a(\o)\,=\,\prod_{i=1}^p\left(\frac{1}{a_i}\int_{\R^+}\o_i^{\beta_i}\,\exp(\o_i/a_i)\,d\o_i\right)\,=\,\a^\beta\,\prod_{i=1}^p\beta_i{\rm !}\,\in\R[\a],\]
and Assumption \ref{ass-on-Ma}(iii) also holds. Like for the Gaussian, and after the change of variable
$u_i=\o_i/a_i$, $i=1,\ldots,p$, one shows easily that Assumption \ref{ass-on-Ma}(iv) holds. Finally
for Assumption \ref{ass-on-Ma-2},
\[\int_\om \exp(c\,\vert\o\vert)\,d\mu_\a(\o)\,=\,
\left(\prod_{i=1}^p \frac{1}{a_i}\right)\,\exp(-\displaystyle\sum_{i=1}^p\o_i(\frac{1}{a_i}-c))\,d\o\,<\,\infty,\]
whenever $c<1/a_i$ for all $i=1,\ldots,p$.

\subsubsection{Mixture of elliptical's}
Assumption \ref{ass-on-Ma}(ii) holds for Example \ref{ex-elliptical}. For instance 
with $\theta:\R_+\to\R_+$ such that $\int_\R t^k\theta(t^2)dt<\infty$ for all $k$, and $s=\int_\R\theta(t^2)dt$, let
\[d\mu_\a(\o)\,=\,\frac{1}{s\,\sigma}\theta((\o-a)^2/\sigma^2)\,d\o,\quad\a=(a,\sigma)\in\A.\]
Then
\[\int\o^jd\mu_\a(\o)\,=\,\frac{1}{s}\int_\R(\sigma t+a)^j \theta(t^2)\,dt\,=\,p_j(a,\sigma),\quad j=0,1,\ldots,\]
\subsubsection{Mixture of Poisson's}

Assumption \ref{ass-on-Ma}(ii) holds for Example \ref{ex-poisson}. For instance 
\[p_1(\a)\,=\,\int_\om \o\,d\mu_\a(\o)\,=\,a;\quad p_2(\a)=\int_\om \o^2\,d\mu_\a(\o)=a^2+a.\]

\subsubsection{Mixture of Binomial's}
Assumption \ref{ass-on-Ma}(ii) holds for Example \ref{ex-binomial}. For instance 
\[p_1(\a)\,=\,\int_\om \o\,d\mu_\a(\o)\,=\,N\,a;\quad p_2(\a)=\int_\om \o^2\,d\mu_\a(\o)=N\,a\,(1-a).\]



\begin{thebibliography}{99}
\bibitem{convexity}
W.  van Ackooij. Convexity Statements for Linear Probability Constraints with Gaussian
Technology Matrices and Copulae Correlated Rows
\bibitem{convexity-2}
W.  van Ackooij, J. Malick. Eventual convexity of probability constraints with elliptical distributions,
Math. Program., pp. 1--27, 2018.
\bibitem{Ash}
R. Ash. {\em Real Analysis and Probability}, Academic Press Inc., Boston, USA, 1972
\bibitem{billingsley}
P. Billingsley. {\em Convergence of Probability Measures}, John Wiley \& Sons, New York, 1968
\bibitem{calafiore1}
G. Calafiore, F. Dabbene, Probabilistic and Randomized Methods
for Design under Uncertainty, G. Calafiore and F. Dabbene (Eds.), Springer, 2006.
\bibitem{calafiore2}
G.C. Calafiore, L. El Ghaoui. On Distributionally Robust Chance-Constrained
Linear Programs, J. Optim. Theory Appl. 130, pp. 1-22, 2006.
\bibitem{opf}
Chao Duan, Wanliang Fang, Lin Jiang, Li Yao, Jun Liu. Distributionally Robust Chance-Constrained
Voltage-Concerned DC-OPF with Wasserstein Metric.
\bibitem{charnes}
A. Charnes, W.W. Cooper. Chance constrained programming. Manag. Sci. 6, pp. 73--79, 1959
\bibitem{chen}
W. Chen, M. Sim, J. Sun, C.P. Teo. From CVaR to uncertainty set: Implications in joint
chance-constrained optimization. Operations research, 58(2), 470-485, 2010
\bibitem{deklerk}
E. de Klerk, D. Kuhn, K. Postek. 
Distributionally robust optimization with polynomial densities: theory, models and algorithms,
{\tt arXiv:1805.03588}
\bibitem{delage}
E. Delage, Y. Ye. Distributionally robust optimization under moment uncertainty with applications
to data-driven problems. Operations Research 58, pp. 595--612, 2010
\bibitem{dizio}
M. Di Zioa, U. Guarneraa, R. Roccib. A mixture of mixture models for a classification problem: The unity measure error, Computational Statistics and Data Analysis 51,  pp. 2573--2585, 2007
\bibitem{doob}
J.L. Doob. {\em Measure Theory}, Springer-Verlag, 1994, New York.
\bibitem{erdogan}
E. Erdogan, G. Iyengar. Ambiguous chance constrained problems and robust optimization. Math. Program. S\'er. B 107, pp. 37--61, 2006
\bibitem{el-ghaoui}
L. El Ghaoui, M. Oks, F. Oustry. Worst-case value-at-risk and robust portfolio optimization: A conic programming approach, Oper. Res. 51, pp. 543--556, 2003
\bibitem{hanasusanto}
G.A. Hanasusanto, V. Roitch, D. Kuhn, W. Wiesemann. Ambiguous joint chance
constraints under mean and dispersion information. Oper. Res. 65, pp. 751-767, 2017
\bibitem{hanasusanto2}
G.A. Hanasusanto, V. Roitch, D. Kuhn, W. Wiesemann. A distributionally robust
perspective on uncertainty quantification and chance constrained programming. Math.
Program. 151, pp. 35-62, 2015.
\bibitem{henrion}
R. Henrion. Structural Properties of Linear Probabilistic Constraints,
Optimization 56, pp. 425--440, 2007.
\bibitem{henrion2}
R. Henrion, C. Strugarek. Convexity of chance constraints with independent random variables. Computational Optimization
and Applications 41, pp. 263--276, 2008.
\bibitem{gloptipoly}
D. Henrion, J.B. Lasserre, J.  Lofberg. Gloptipoly 3: moments, optimization and semidefinite programming,  Optim. Methods and Softwares 24,  pp. 761--779, 2009.
\bibitem{sirev}
D. Henrion, J. B. Lasserre, C. Savorgnan, Approximate volume and integration for basic semi-algebraic sets, SIAM Review 51, pp. 722--743, 2009.
\bibitem{korda1}
D. Henrion, M. Korda, Convex computation of the region of attraction of polynomial control systems. IEEE Trans. Aut. Control 59, pp. 297--312, 2014.
\bibitem{lagoa1}
A. M. Jasour, N.S. Aybat, C. M. Lagoa, Semidefinite programming for chance constrained optimization over semi-algebraic sets, SIAM J. Optim. 25, No 3, pp. 1411--1440, 2015
\bibitem{lagoa2}
A. Jasour, C. Lagoa, Convex constrained semialgebraic volume optimization: Application in systems and control, 2017, {\tt arXiv:1701.08910}
\bibitem{jiang}
R. Jiang, Y. Guan. Data-driven chance constrained stochastic program. Math. Program. 158, 291-327, 2016
\bibitem{korda2}
M. Korda, D. Henrion, C. N. Jones, Convex computation of the maximum controlled invariant set for polynomial control systems, SIAM J. Control Optim. 52(5):2944-2969, 2014.
\bibitem{lass-ieee}
J.B. Lasserre, Representation of chance-constraints with strong asymptotic properties,
IEEE Control Systems Letters 1,  pp. 50--55, 2017.
\bibitem{lass-tams}
J.B. Lasserre, The K-moment problem for continuous linear functionals,
Trans. Amer. Math. Soc. 365, pp. 2489--2504, 2013. 
\bibitem{lass-siopt-2010}
J.B. Lasserre, A ``Joint+Marginal" approach to parametric polynomial optimization,
SIAM J. Optim. 20, pp. 1995--2022, 2010.
\bibitem{book1}
J. B. Lasserre, Moments, positive polynomials and their applications. Imperial College Press, London, 2010.
\bibitem{lebesgue}
J.B. Lasserre, Lebesgue decomposition in action via semidefinite relaxations,
Adv. Comput. Math. 42,  pp. 1129--1148, 2016.
\bibitem{lass-netzer}
J.B. Lasserre, T. Netzer. SOS approximations of nonnegative polynomials via simple high degree perturbations".  Math. Zeitschrift 256,  pp. 99--112, 2006.
\bibitem{gaussian}
J.B. Lasserre, Computing gaussian and exponential measures of semi-algebraic sets,
Adv. Appl. Math. 91,  pp. 137--163, 2017.
\bibitem{li}
P. Li, M. Wendt, G. Wozny, A Probabilistically Constrained Model
Predictive Controller, Automatica 38, pp. 1171--1176, 2002.
\bibitem{marron}
J.S. Marron, M.P. Wand. Exact mean integrated squared error, The Ann. Statist. 20, pp. 712--736, 1992
\bibitem{miller}
B. Miller, H. Wagner. Chance-constrained programming with joint constraints. Oper. Res. 13, pp. 930--945, 1965
34. 
\bibitem{mosek}
Mosek\;Aps, Mosek Matlab Toolbox, 2017.
\bibitem{nemirovski}
A. Nemirovski, A. Shapiro. Convex approximations of chance constrained programs. SIAM J. Optim. 17, pp. 969--996, 2006
\bibitem{prekopa}
A. Pr\'ekopa, Probabilistic Programming, in Stochastic Programming,
A. Ruszczynski and A. Shapiro (Eds.), Handbooks in Operations
Research and Management Science Volume 10, pp. 267--351, 2003.
\bibitem{putinar}
M. Putinar, Positive polynomials on compact semi-algebraic sets, Indiana Univ. Math. J. 42, pp. 969--984, 1993.
\bibitem{shapiro}
A. Shapiro, D. Dentcheva, A. Ruszczynski, Lecture on Stochastic
Programming: Modeling and Theory, 2nd Ed., SIAM, Philadelphia, 2014.
\bibitem{wang}
S. Wang, J. Li, C. Peng. Distributionally robust chance-constrained program surgery planning with downstream resource,
Proceedings 2017 International Conference on Service Systems and Service Management, Dalian, China, June 2017.
\bibitem{waki}
H. Waki, S. Kim, M. Kojima, M. Muramatsu, Sums of squares and semidefinite programming relaxations for polynomial optimization problems with structured sparsity, SIAM J. Optim. 17, pp. 218--242, 2006.
\bibitem{wang2}
S.I. Wang, A.T. Chaganti, P. Liang. Estimating Mixture Models via Mixtures of Polynomials
Proceedings of the 28th International Conference on Neural Information Processing Systems (NIPS'15), Montreal, December 2015, pp. 487--495 
\bibitem{jogo}
X. Tong, H. Sun, X. Luo, Q. Zheng, Distributionally robust chance 
constrained optimization for economic dispatch in renewable energy integrated systems,
J. Global Optim. 70, pp. 131--158, 2018.
\bibitem{xie}
W. Xie, S. Ahmed.  Distributionally robust chance constrained optimal power
flow with renewables: A conic reformulation, IEEE Trans. Power Systems, 2017.
\bibitem{xie2}
W. Xie, S. Ahmed. On deterministic reformulations of distributionally robust joint chance
constrained optimization problems. SIAM J. Optim. 28, pp. 1151-1182, 2018.
\bibitem{survey}
D. Xu. The Applications of Mixtures of Normal Distributions in Empirical Finance: A Selected Survey, Working Papers 0904, University of Waterloo, Department of Economics, revised September 2009.
\bibitem{yang}
Wenzhuo Yang, Huan Xu. Distributionally Robust Chance Constraints for Non-Linear Uncertainties,
Math. Program. 155, pp. 231--265, 2016.
\bibitem{zhang}
Zhang, Y., Shen, S., J. Mathieu. Distributionally robust chance-constrained optimal
power flow with uncertain renewables and uncertain reserves provided by loads, IEEE
Trans. Power Systems 32, pp. 1378--1388, 2016.
\bibitem{zymler}
S. Zymler, D. Kuhn, B. Rustem. Distributionally robust joint chance constraints with
second-order moment information. Math. Program. 137, pp. 167-198, 2013

\end{thebibliography}
\end{document}